\documentclass[12pt,a4paper]{amsart}
\usepackage[utf8]{inputenc}
\usepackage[T1]{fontenc}
\usepackage[english]{babel}
\usepackage{amsmath}
\usepackage{amsfonts}
\usepackage{amssymb}
\usepackage{amsthm}
\usepackage{mathtools}
\usepackage{tabularx}
\usepackage{adjustbox}
\usepackage{graphicx}
\usepackage{tikz}
\usetikzlibrary{calc}
\usetikzlibrary{patterns.meta}

\usepackage{fullpage}
\usepackage{indentfirst}
\usepackage{fancyvrb}
\usepackage{comment}
\usepackage{lmodern}
\usepackage{xcolor}
\usepackage{listings}
\usepackage{url}
\usepackage{float}
\usepackage{diagbox}
\usepackage{hyperref}
\numberwithin{equation}{section}

\theoremstyle{plain}
\newtheorem{thm}{Theorem}[section]

\newtheorem{lm}[thm]{Lemma}

\newtheorem{cor}[thm]{Corollary}

\newtheorem{prob}[thm]{Problem}

\theoremstyle{definition}
\newtheorem{df}[thm]{Definition}
\newtheorem{ex}[thm]{Example}

\theoremstyle{remark}
\newtheorem{re}[thm]{Remark}

\renewcommand{\arraystretch}{1.3}

\DeclareRobustCommand{\svdots}{%
  \vbox{%
    \baselineskip=0.33333\normalbaselineskip
    \lineskiplimit=0pt
    \hbox{.}\hbox{.}\hbox{.}%
    \kern-0.2\baselineskip
  }%
}

\makeatletter
\let\@@pmod\pmod
\DeclareRobustCommand{\pmod}{\@ifstar\@pmods\@@pmod}
\def\@pmods#1{\mkern4mu({\operator@font mod}\mkern 6mu#1)}
\makeatother

\title{A note on the partition function of a rectangle}

\author{Krystian Gajdzica}

\address{Theoretical Computer Science Department \\ Faculty of Mathematics and Computer Science\\ Jagiellonian University\\ Łojasiewicza 6\\ 30-348 Kraków\\ Poland}
\email{krystian.gajdzica@uj.edu.pl}

\author{Maciej Zakarczemny}

\address{Department of Applied Mathematics, Faculty of Computer Science and Mathematics, Cracow University of Technology, Warszawska 24, 31-155 Krak\'ow, Poland}

\email{maciej.zakarczemny@pk.edu.pl}



\begin{document}

\setlength{\parindent}{10mm}

\begin{abstract}
We investigate the asymptotic behavior of the rectangle partition functions $p(n,n)$ and $p(3,n)$. The function $p(n,n)$ counts partitions of the square $n\times n$ into rectangular blocks with integer sides, while $p(3,n)$ counts such partitions of the rectangle $3\times n$. Two rectangle partitions are identified when they contain the same multiset of rectangle types, and a block $a\times b$ is identified with a block $b\times a$. Our main results are
$$
p(n,n)
=
\exp\left(
\left(\tfrac{\pi}{\sqrt3}+o(1)\right)n\sqrt{\log n}
\right)
$$
and
$$
p(3,n)
=
\exp\left(
\pi\sqrt{\tfrac{11n}{3}}+O(\log n)
\right).
$$
We also present simpler upper and lower bounds for $p(n,n)$ and an independent upper bound for $p(3,n)$.
\end{abstract}




\maketitle

\section{Introduction}
We study a generalization of integer partitions, called rectangle partitions.\\ Before defining rectangle partitions, recall that a partition of a non-negative integer $n$ is a non-increasing sequence $\lambda=(\lambda_1,\lambda_2,\ldots,\lambda_r)$ of positive integers such that\\
$|\lambda|:=\lambda_1+\lambda_2+\cdots+\lambda_r=n.$ 
Elements $\lambda_i$ are called the parts of the partition $\lambda$. The number of all such sequences for $n$ is denoted by $p(n)$, and refers to the partition function. The theory of integer partitions plays an important role in mathematics and other sciences. Its history goes back to Euler, who, among other things, discovered the generating function of the partition function $p(n)$ and proved that the number of partitions of $n$ into distinct parts is equinumerous with the number of partitions of $n$ into odd parts. Today, there is an abundance of literature devoted to the theory of partitions. In particular, asymptotic estimates, divisibility properties and the so-called log-behaviour of various partition statistics enjoy great popularity these days. For a gentle introduction to the topic, we refer the reader to Andrews' classic books \cite{Andrews1, Andrews2}.

We investigate a generalization of the partition function recently introduced by the authors and Robin Visser \cite{GVZ2026}. Namely, we study the number of partitions $p(m,n)$ of a rectangle $m\times n$ into integer-sided rectangular blocks, where two rectangular partitions are indistinguishable if they consist of the same multiset of blocks (the geometric arrangement of these blocks is irrelevant). For the sake of clarity, let us recall the precise definition of a partition of a rectangle from \cite{GVZ2026}.
Throughout the paper, all rectangles are assumed to have sides parallel to the coordinate axes. 

\begin{df}\label{def: rectangle partition}
Let $R = [0, m] \times [0, n] \subset \mathbb{R}^2$, where $m, n \in \mathbb{Z}_+$. A~partition of the rectangle $R$ is defined as a~multiset of rectangles
$\mathcal{P} = \{ R_i \}_{i \in I}$ such that:
\begin{enumerate}
    \item each $R_i$ is an integer-sided rectangle;
    \item rectangles of sizes $k\times l$ and $l\times k$ are indistinguishable;  
    \item all the rectangles from $\mathcal{P}$ can be arranged in such a~way that they cover exactly the area of the rectangle $R$ and their interiors do not intersect. 
\end{enumerate}
 Moreover, two partitions of a rectangle are considered to be the same if they are equal as multisets. The number of partitions of a rectangle of size $ m \times n $ is denoted by $p(m,n)$.
\end{df}

To illustrate the above definition, we exhibit a few examples. 


\begin{ex} We have that $p(3,3) = 21$. The following figure shows one representative tiling for each of the $p(3,3)$ multisets of rectangle types that tile the $3\times3$ square:
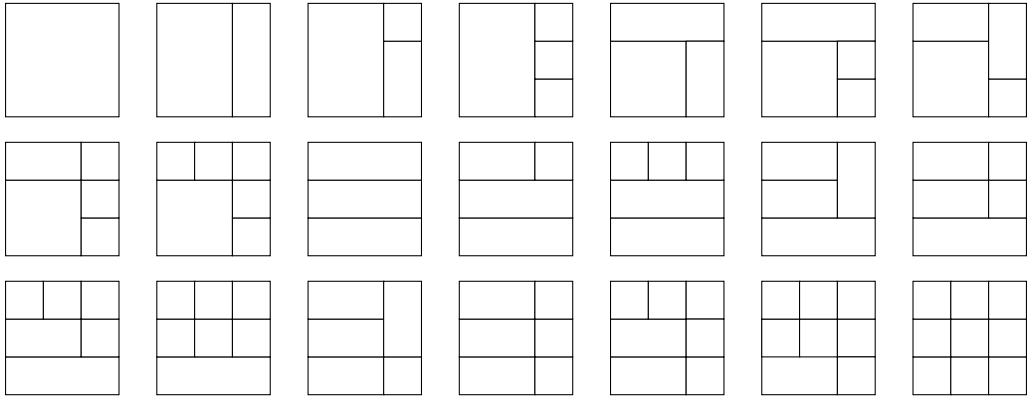
\begin{align*}

\end{align*}
\end{ex}

\begin{ex} We have that $p(4,4) = 192$.
The following figure shows one representative tiling for each of the $p(4,4)$ multisets of rectangle types that tile the $4\times4$ square:
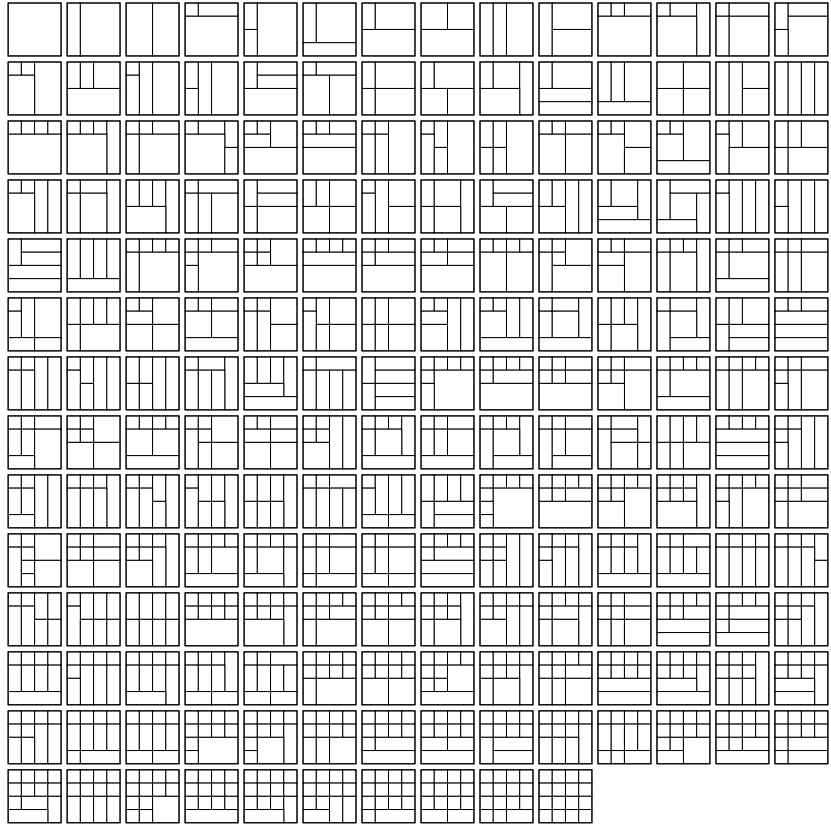
\begin{center}
%
\end{center}
\end{ex}
It is easy to see that $p(1,n)=p(n)$ and $p(m,n)=p(n,m)$. We know the exact values of $p(2,n)$ for $n\le 40$ and $p(3,n)$ for $n\le 20$, see Sequences A360631 \cite{OEISA360631} and A360632 \cite{OEISA360632} in the OEIS, respectively. There is also a special case of a sequence of square partitions $p(n,n)$, which takes the values
$$
1,\ 4,\ 21,\ 192,\ 2035,\ 27407,\ 399618
$$
for $1\le n\le 7$, see Sequence A360630 \cite{OEISA360630} in the OEIS.
    Our own computer computations also show that $p(8,8)=6501934.$


Recently, the authors together with Robin Visser obtained the asymptotic formula for $p(2,n)$, see \cite[Theorem 1.5]{GVZ2026}:
$$
p(2,n)\sim
\frac{\pi\sqrt[4]{2}}{32n^{\tfrac{7}{4}}}
\exp\!\left(\pi\sqrt{2n}\right),
\hspace{0.5cm}\text{as }n\to\infty.
$$
That result can be viewed as an analogue of the Hardy–Ramanujan asymptotic formula for the partition function $p(n)$ \cite{HardyRamanujan}, stating that
$$
p(n)\sim
\frac{1}{4\sqrt{3}n}
\exp\!\left(\pi\sqrt{\tfrac{2}{3}n}\right),
\hspace{0.5cm}\text{as }n\to\infty.
$$
In fact, we conjecture a stronger result. More precisely, numerical evidence and heuristic arguments suggest that, for every fixed positive integer $m$,
\begin{equation}\label{conjH}
p(m,n)=\exp\!\left(
\pi\sqrt{\tfrac{2}{3}mH_m\,n}+o(\sqrt n)
\right),\hspace{0.5cm}\text{as }n\to\infty,
\end{equation}
where
$
H_m=
1+\frac12+\cdots+\frac1m
$
is the $m$-th harmonic number. For $m=1$, the formula follows from the Hardy--Ramanujan asymptotic. For $m=2$, it follows from the asymptotic formula proved in \cite[Theorem~1.5]{GVZ2026}. Theorem~\ref{thm: asymptotic p(3,n)} proves the formula for $m=3$ with the stronger error term $O(\log n)$ in the logarithm. Thus, the conjecture remains open for every fixed integer $m\ge 4$.\\
There is also an intriguing question that does not fit into the above asymptotic formulae. Namely, what is the precise asymptotic behaviour of $p(n,n)$? The square case leads to a~different asymptotic problem, since both side lengths tend to infinity. In particular, one may ask for the exact multiplicative factor in the asymptotic formula.\\
We determine the main asymptotic terms of both $p(n,n)$ and $p(3,n)$. Before stating the~asymptotic formula for $p(n,n)$, we present two weaker upper bounds and an elementary lower bound. We retain these estimates because their proofs are considerably less technical and provide independent approaches to the square partition problem. Theorem~\ref{thm: p(n,n) UB} is proved directly, while Theorem~\ref{p(n,n)stronger} uses a result of Berndt, Robles, Zaharescu and Zeindler related to the number of partitions of $n$ into parts from the multiset $\{1,2_1,2_2,3_1,3_2,4_1,4_2,4_3,\ldots\}$, where each number $j$ occurs $d(j)$ times; see \cite{BerndtRoblesZaharescuZeindler}.

\begin{thm}\label{thm: p(n,n) UB}
As $n\to\infty$, we have
$$
p(n,n)
\le
\exp\left(
\sqrt2\,n\log n+\sqrt2\,\gamma n
+O\!\left(\tfrac{n}{\log n}\right)
\right),
$$
where $\gamma\approx 0.5772$ is the Euler--Mascheroni constant.
\end{thm}

\begin{thm}\label{p(n,n)stronger} 
As $n\to\infty$, we have
 $$p(n,n)\le \exp\left(\left(\pi\sqrt{\tfrac23}+o(1)\right)n\sqrt{\log n}\right).$$
 \end{thm}
Apart from the upper estimates, we also show an elementary lower bound for the square partition function $p(n,n)$.

\begin{thm}\label{thm: p(n,n) LB}
    For a positive integer $n$, we have that
    $$p(n,n)\ge 5^{\lfloor n/2\rfloor-2}.$$
\end{thm}
The proof of the exact asymptotic formula is more technical. It combines colored partitions, uniform estimates for restricted ordinary partitions and a packing construction. The main result concerning square partitions is the following. 
\begin{thm}\label{asymp(n,n)}
As $n\to\infty$, we have
$$
p(n,n)
=
\exp\left(
\left(\tfrac{\pi}{\sqrt3}+o(1)\right)n\sqrt{\log n}
\right).
$$
\end{thm}

For $p(3,n)$, we first derive an independent upper bound, which is weaker than the estimate used in the proof of Theorem~\ref{thm: asymptotic p(3,n)}. Nevertheless, we retain it as an alternative analytic approach. 


\begin{thm}\label{thmp(3n)UB}
If $n\to\infty$, then
$$
p(3,n)
\le
\tfrac{11}{432}\,
n^{-\frac{3}{2}}
\exp\left(\pi\sqrt{\tfrac{11n}{3}}\right)
(1+o(1)).
$$
In particular, we have that
$$
p(3,n)
\le
\exp\left(
\pi\sqrt{\tfrac{11n}{3}}+o(\sqrt n)
\right),\hspace{0.5cm}\text{as }n\to\infty.
$$
\end{thm}

We next give the exact asymptotic formula for $p(3,n)$.

\begin{thm}\label{thm: asymptotic p(3,n)}
As $n\to\infty$, we have
$$
p(3,n)
=
\exp\left(
\pi\sqrt{\tfrac{11n}{3}}+O(\log n)
\right).
$$
\end{thm}

The paper is organized as follows. Section 2 contains conventions, definitions and auxiliary results used throughout the paper. Section 3 contains the proofs of the two upper bounds for $p(n,n)$ stated in Theorems~\ref{thm: p(n,n) UB} and~\ref{p(n,n)stronger}. Section 4 gives an elementary proof of the lower bound from Theorem~\ref{thm: p(n,n) LB}. Although the estimates proved in Sections 3 and 4 are weaker than the final asymptotic formula, we retain them because their proofs are less technical. Section 5 contains the proof of Theorem~\ref{asymp(n,n)}, while Section 6 gives an independent upper bound for $p(3,n)$ and establishes auxiliary asymptotic estimates used later. Finally, Section 7 is devoted to the proof of Theorem~\ref{thm: asymptotic p(3,n)}.

\section{Preliminaries}
Throughout the paper we use the following conventions and notations. By $\mathbb{R}$, $\mathbb{Z}$, $\mathbb{Z}_{\ge 0}$ and $\mathbb{Z}_+$, we mean the set of real numbers, the set of integers, the set of non-negative integers and the set of positive integers, respectively. The $n$-th harmonic number $H_n$ is the sum of the reciprocals of the first $n$ positive integers:
$$H_n:=1+\frac{1}{2}+\cdots+\frac{1}{n}.$$
The notion defined above can be extended to the so-called higher order harmonic numbers. For a positive integer $r>1$, the $r$-th order harmonic number $H_n^{(r)}$ is the sum of the reciprocals raised to the $r$-th power of the first $n$ positive integers:
$$H_n^{(r)}:=1+\frac{1}{2^r}+\cdots+\frac{1}{n^r}.$$
By $d(n)$, we denote the number of positive integer divisors of $n$. For a given logic sentence $P$, the indicator function  $\mathbf 1_{\{P\}}$ is defined by
$$
\mathbf 1_{\{P\}}:=
\begin{cases}
1,& P\text{ is true},\\
0,& P\text{ is false}.
\end{cases}
$$

The Dedekind eta function $\eta(\tau)$ is defined in the half-plane $\mathbb{H}=\{\tau:\operatorname{Im}(\tau)>0\}$ by the equation
\begin{align*}
    \eta(\tau):=e^{\pi i\tau/12}\prod_{n=1}^\infty\left(1-e^{2\pi in\tau}\right).
\end{align*}
More information concerning the Dedekind eta function can be found in Apostol's book \cite[Chapter 3]{ApostolModularFunctions}. For two sequences $a(n)$ and $b(n)$, we write $a\asymp b$ if and only if there exist positive real numbers $C$ and $D$ such that $C|b(n)|\le  |a(n)|\le  D|b(n)|$. For the definitions of a partition, the partition function $p(n)$, a partition of a rectangle and the rectangular partition function $p(m,n)$, we refer the reader to Introduction.\\
We introduce an auxiliary concept that will be useful to prove Theorem \ref{thm: p(n,n) UB}. It is based on the fact that in Definition \ref{def: rectangle partition}, a block $a\times b$ is identified with a block $b\times a$.\\
For a positive integer $v$, let $P_v(\ell)$ be the number of multisets of rectangle types $a\times b$, where $1\le a\le b\le v$, whose total area is $\ell$. The geometric realizability condition (the condition (3) of Definition \ref{def: rectangle partition}) is ignored. Further, for a positive integer $s$, let us put 
$$r_v(s)=|\{(a,b)\in\mathbb{Z}_+^2:\ a\le b\le v,\ ab=s\}|.$$
It follows that
\begin{align}\label{def: generating P_v(n)}
\sum_{\ell=0}^\infty P_v(\ell)x^\ell=\prod_{s=1}^{v^2}\frac{1}{(1-x^s)^{r_v(s)}}=\prod_{1\le a\le b\le v}\frac{1}{1-x^{ab}}.    
\end{align}
All the coefficients $P_v(\ell)$ are non-negative. Moreover, it turns out that 
\begin{equation}\label{ineq:pnn-Pn}
p(n,n)\le P_n(n^2).
\end{equation} 
Indeed, every partition of the square $n\times n$ gives a multiset of admissible rectangle types with total area $n^2$. Conversely, not every multiset of rectangles with total area $n^2$ can be rearranged to tile the square $n\times n$.
For instance, let us consider two rectangles: $3\times3$ and $4\times 4$. Their total area is $25$, but they can not exactly cover the square $5\times5$.\\
We can now proceed to the proofs of the theorems.

\section{The proofs of Theorem \ref{thm: p(n,n) UB} and Theorem \ref{p(n,n)stronger}}

In this section, we prove Theorems \ref{thm: p(n,n) UB} and \ref{p(n,n)stronger}. Our reasoning relies on ignoring the condition (3) of Definition \ref{def: rectangle partition}, as it was described in Preliminaries.

\begin{proof}[Proof of Theorem \ref{thm: p(n,n) UB}]
Let us fix $t>0$. Since $e^{-t}\in(0,1)$, we can set $x=e^{-t}$ in \eqref{def: generating P_v(n)} to obtain
$$
\prod_{1\le a\le b\le v}\tfrac{1}{1-e^{-tab}}
=
\sum_{\ell=0}^\infty P_v(\ell)e^{-t\ell}
\ge
P_v(n^2)e^{-tn^2}.$$
Setting $v=n$, it follows that
\begin{align}\label{ineq: P_n(n^2) UP (1)}
P_n(n^2)\le
e^{tn^2}
\prod_{1\le a\le b\le n}\frac{1}{1-e^{-tab}}=\exp\left(
n^2t+
\sum_{1\le a\le b\le n}
\log\frac{1}{1-e^{-tab}}
\right).
\end{align}

Let $u>0$. The inequalities $e^u\geq1+u$ and $e^{1/u}\geq1+1/u$ give
$
1-e^{-u}\geq\tfrac{u}{u+1}
$
and
$
e^{-1/u}\le \tfrac{u}{u+1},
$
respectively. Hence,
$
1-e^{-u}\geq e^{-1/u},
$
and therefore
$
\log\tfrac{1}{1-e^{-u}}\le \tfrac1u.
$\\
Setting $u=tab$ implies that
\begin{align*}
    \sum_{1\le a\le b\le n}
\log\frac{1}{1-e^{-tab}}
\le
\frac1t
\sum_{1\le a\le b\le n}\frac1{ab}.
\end{align*}
For the sake of convenience, let us put
$$S_n=\sum_{1\le a\le b\le n}\tfrac1{ab}=\tfrac12\left(
\sum_{a=1}^{n}\sum_{b=1}^{n}\tfrac1{ab}
+
\sum_{a=1}^{n}\tfrac1{a^2}
\right)
=
\tfrac12H_n^2+\tfrac12H_n^{(2)},
$$
where $H_n$ and $H_n^{(2)}$ are the $n$-th harmonic number and the $n$-th second order harmonic number, respectively. Thus, we can write that
\begin{align}\label{ineq: S_n}
    \sum_{1\le a\le b\le n}
\log\frac{1}{1-e^{-tab}}
\le
\frac1t
\sum_{1\le a\le b\le n}\frac1{ab}=\frac{1}{t}S_n.
\end{align}
Combining both \eqref{ineq: P_n(n^2) UP (1)} and \eqref{ineq: S_n} gives us that
\begin{equation}\label{ineq:Pn-basic}
P_n(n^2)\le
\exp\left(
n^2t+\frac{S_n}{t}
\right).
\end{equation}
Since we want to make the right hand side of \eqref{ineq:Pn-basic} as small as possible, we compute the minimum of the function $n^2t+\frac{S_n}{t}$ with respect to $t$. It is attained at $t=\frac{\sqrt{S_n}}{n}$. Next, we use the well-known estimates 
\begin{align*}
H_n=\log n+\gamma+O\!\left(\frac1n\right)\hspace{0.5cm}\text{and}\hspace{0.5cm}H_n^{(2)}=\frac{\pi^2}{6}+O\!\left(\frac1n\right),    
\end{align*}
where $\gamma\approx0.5772$ is the Euler--Mascheroni constant. It follows that
\begin{equation}\label{eq:Sn-asymptotic}
S_n
=
\tfrac12H_n^2+\tfrac12H_n^{(2)}
=
\tfrac12(\log n)^2+\gamma\log n
+\tfrac12\left(\gamma^2+\tfrac{\pi^2}{6}\right)
+O\!\left(\tfrac{\log n}{n}\right).
\end{equation}
If we put $t=\frac{\sqrt{S_n}}{n}>0$ and use \eqref{eq:Sn-asymptotic}, we obtain
$$
P_n(n^2)
\le
\exp\left(
2n\sqrt{S_n}
\right)
=
\exp\left(
\sqrt2\,n\log n+\sqrt2\,\gamma n
+O\!\left(\tfrac{n}{\log n}\right)
\right).
$$
Finally, we know that $p(n,n)\le  P_n(n^2)$ which gives us
$$
p(n,n)
\le 
\exp\left(
\sqrt2\,n\log n+\sqrt2\,\gamma n
+O\!\left(\tfrac{n}{\log n}\right)
\right),
$$
or, equivalently,
$
p(n,n)\le n^{\sqrt2\,n+\sqrt2\,\gamma\frac{ n}{\log n}+O\!\left(\frac{n}{(\log n)^2}\right)}.$ This concludes the proof.
\end{proof}
There are some remarks arising from the above proof.
\begin{re}
The argument from the proof of Theorem \ref{thm: p(n,n) UB}, together with the more accurate inequalities
$
    H_n\le \log n+\gamma+\frac1{2n}$ and $H_n^{(2)}\le \frac{\pi^2}{6},$
gives us the explicit upper bound of $p(n,n)$. Namely, the inequality
$$
p(n,n)
\le
n^{
\sqrt2\,\frac{n}{\log n}
\sqrt{
\left(\log n+\gamma+\frac1{2n}\right)^2
+\frac{\pi^2}{6}
}
}
$$
is satisfied for every $n\ge 2$.
\end{re}

\begin{re}
In fact, the same proof gives the stronger bound
$$
p(n,n)
\le
\exp\!\left(
\min_{t>0}
\left(
n^2t+
\sum_{1\le a\le b\le n}
\log\frac{1}{1-e^{-abt}}
\right)
\right).
$$
However, there is no simple closed formula for the minimum on the right hand side of the above inequality.
\end{re}
We now turn to the proof of Theorem \ref{p(n,n)stronger}.
\begin{proof}[Proof of Theorem \ref{p(n,n)stronger}] 
For the divisor function $d(s)$, let $D(\ell)$ be defined by $$\sum_{\ell=0}^{\infty}D(\ell)x^\ell=\prod_{s=1}^{\infty}\frac{1}{(1-x^s)^{d(s)}}.$$ 
Berndt, Robles, Zaharescu and Zeindler proved that $\log D(\ell)\sim\frac{\pi}{\sqrt3}\sqrt{\ell\log\ell}$ as $\ell\to\infty$; see \cite[Equations (4.13) and~(4.14)]{BerndtRoblesZaharescuZeindler} and Remark \ref{blad}. 
Recall that $r_n(s)$ is the number of rectangle types $a\times b$ satisfying $1\le  a\le  b\le  n$ and $ab=s$. Since every such pair corresponds to a positive divisor of $s$, we have $r_n(s)\le  d(s)$ for every positive integer $s$. All coefficients of the considered products are non-negative. Therefore, comparison of the coefficients in $\sum_{\ell\geq0}P_n(\ell)x^\ell=\prod_{1\leq s\leq n^2}(1-x^s)^{-r_n(s)}$ and $\sum_{\ell\geq0}D(\ell)x^\ell=\prod_{s\geq1 }(1-x^s)^{-d(s)}$ gives $P_n(n^2)\le  D(n^2).$ Furthermore, we have that  $$\log D(n^2)\sim\tfrac{\pi}{\sqrt3}\sqrt{n^2\log(n^2)}=\pi\sqrt{\tfrac{2}{3}}\,n\sqrt{\log n}.$$ Hence, it follows that $$P_n(n^2)\le \exp\left(\left(\pi\sqrt{\tfrac{2}{3}}+o(1)\right)n\sqrt{\log n}\right).$$ Finally, the inequality $p(n,n)\le  P_n(n^2)$ gives $$p(n,n)\le \exp\left(\left(\pi\sqrt{\tfrac{2}{3}}+o(1)\right)n\sqrt{\log n}\right),$$ as required. 
\end{proof}

\begin{re}\label{blad}
The constant in $\log D(\ell)\sim\frac{\pi}{\sqrt3}\sqrt{\ell\log\ell}$ as $\ell\to\infty$ follows from Equations (4.13) and (4.14) of \cite{BerndtRoblesZaharescuZeindler}, with $k_1=k_2=0$ and $p_{0,0}(n)=D(n)$ according to the notation of \cite{BerndtRoblesZaharescuZeindler}, the derivation requires careful technical calculations. We can not directly apply Theorem~1.1 of \cite{BerndtRoblesZaharescuZeindler}, because of a typographical error. The correct coefficient $\frac{\pi}{\sqrt3}$ is also confirmed by numerical computations based on the convenient form of the generating function of $D(\ell)$.
\end{re}

At this point, we make some numerical computations in order to illustrate the differences between $P_n(n^2)$ and $p(n,n)$. We start by considering tilings of the square of size $5\times5$.

\begin{ex}
From Introduction (or Table \ref{tblp(n,n)vsP_n(n^2)}), we know that the number of all possible partitions of the square $5\times5$ is $2035$, i.e. $p(5,5)=2035$. On the other hand, all the possible rectangle types that could be used to tile the square are listed below.

\begin{table}[htb]
    \renewcommand{\arraystretch}{1.4}
\begin{tabular}{ |c||c|c|c|c|c|c|c| }
 \hline
 \text{Area} & $1$ & $2$ & $3$ & $4$ & $5$ & $6$ & $8$\\\hline
 \text{Size} & $1\times1$ & $1\times2$ & $1\times3$ & $1\times4,2\times2$ & $1\times5$ & $2\times3$ & $2\times4$\\\hline\hline
 \text{Area} & $9$ & $10$ & $12$ & $15$ & $16$ & $20$ & $25$\\\hline
 \text{Size} & $3\times3$ & $2\times5$ & $3\times4$ & $3\times5$ & $4\times4$ & $4\times5$ & $5\times5$\\
 \hline
\end{tabular}
\medskip
\caption{All blocks that can be used in a partition of the square $5\times5$, grouped by their areas.}\label{tbl5x5}
\end{table}
For the sake of convenience, let us put $S:= \{1,2,3,5,6,8,9,10,12,15,16,20,25\}.$ In such a setting, we have that
$$
r_5(s)=
\begin{cases}
2,& s=4,\\
1,& s\in S,\\
0,& \text{otherwise}.
\end{cases}
$$
Therefore, we can write the generating function of $P_5(\ell)$ as follows:
\begin{align*}
\sum_{\ell=0}^\infty P_5(\ell)x^\ell
=\frac{1}{(1-x^4)^2}\prod_{s\in S}\frac{1}{1-x^s}.
\end{align*}
One can check that $P_5(25)=2207$, and, in conclusion, $p(5,5)=2035<2207=P_5(25)$.
\end{ex}

Before we proceed to the next part of the paper, let us present a table comparing values of both $p(n,n)$ and $P_n(n^2)$ for small values of $n$.

\begin{table}[htb]
    \renewcommand{\arraystretch}{1.4}
\begin{tabular}{ c|c|c|c|c }
 $n$ & $p(n,n)$ & $P_n(n^2)$ & $P_n(n^2)-p(n,n)$ & $p(n,n)/ P_n(n^2)$\\\hline
  $2$ & $4$ & $4$ & $0$ & $1$\\
  $3$ & $21$ & $22$ & $1$ & $0.954545$\\
  $4$ & $192$ & $205$ & $13$ & $0.936585$\\
  $5$ & $2035$ & $2207$ & $172$ & $0.922066$\\
  $6$ & $27407$ & $28806$ & $1399$ & $0.951434$\\
  $7$ & $399618$ & $416200$ & $16582$ & $0.960159$\\
  $8$ & $6501934$ & $6645769$ & $143835$ & $0.978357$ \\
\end{tabular}
\medskip
\caption{The values of $p(n,n)$, $P_n(n^2)$, $P_n(n^2)-p(n,n)$ and $p(n,n)/P_n(n^2)$ for $2\le  n\le 8$.}
\label{tblp(n,n)vsP_n(n^2)}
\end{table}

The values of $p(n,n)$ for $2\le  n\le 7$ are taken from \cite{GVZ2026,OEISA360630}. The remaining entries were obtained by the authors' own calculations. The data suggest that the quotient $\tfrac{p(n,n)}{P_n(n^2)}$ may tend to $1$ as $n\to\infty$. If this were true, then we would have that $p(n,n)\sim P_n(n^2)$. The proof of Theorem~\ref{asymp(n,n)} gives the weaker relation
$
\log P_n(n^2)
\sim
\log p(n,n)
\sim
\tfrac{\pi}{\sqrt3}n\sqrt{\log n}.
$
Only the first few values of $p(n,n)$ are known, and they may accidentally support the stronger conjecture.

\section{The proof of Theorem \ref{thm: p(n,n) LB}}

This section is devoted to the proof of Theorem \ref{thm: p(n,n) LB}. Before we proceed to the main part, some preparation is needed. We start by showing the following auxiliary result.

\begin{lm}\label{lm: pack nxn}
Let $n\geq 2$ and $m\le \tfrac{n}{2}$ be fixed. Suppose that a multiset of integer-sided rectangles of height $1$ and length at most $m$ has total length $T$. If we have that
$$
T\le  (n-1)(n-m),
$$
then these rectangles can be packed into an $n\times n$ square.
\end{lm}

\begin{proof}
It suffices to apply the following greedy algorithm. Start from the bottom row of the square $n\times n$ and pack the rectangles from left to right until the next rectangle does not fit. Then move up to the next row and repeat the procedure. 
At the end, every completed row (except possibly the last one) has occupied area greater than $1\times (n-m)$. Indeed, if a row was filled to length at most $n-m$, then a~rectangle of length at most $m$ would still fit into it.
Let us suppose that the greedy procedure uses $q$ rows. Then the first $q-1$ rows are completed, which implies that
$$
(q-1)(n-m)<T.
$$
Since $T\le  (n-1)(n-m)$, we get that $q-1<n-1$, and hence $q< n$. Therefore, all the rectangles fit into the square $n\times n$, as required.
\end{proof}
We move to the main part of this section.
\begin{proof}[Proof of Theorem \ref{thm: p(n,n) LB}] For $1\le n\le 5$, we check the inequality directly.
Let us fix $n\geq6$ and put $m=\left\lfloor n/2\right\rfloor-1$. For every $k\in\{2,3,\ldots,m\}$, we choose independently a number
$c_k\in\{0,1,2,3,4\}.$ The number $c_k$ indicates how many times the block $1\times k$ is chosen to a multiset of rectangles. Hence, the total area of the constructed multiset is
$$T=\sum_{k=2}^{m}k c_k \le  4\sum_{k=2}^{m}k = 4\left(\frac{m(m+1)}2-1\right) = 2m(m+1)-4. $$
Since $m=\lfloor n/2\rfloor-1$, we have that
$$2m(m+1)-4\le  (n-1)(n-m).$$
Thus, Lemma \ref{lm: pack nxn} guarantees the considered multiset of rectangles can be packed into the~square $n\times n$. Let us do that and fill the remaining area with unit squares $1\times1$, obtaining a partition of the square $n\times n$.\\
It suffices to notice that every choice of the numbers $c_k$ for $k\in\{2,3,\ldots,m\}$ determines a unique partition of the $n\times n$ square. Indeed, all of these numbers $c_k$ can be recovered from the resulting tiling of the square by looking at the multiplicities of the non-unit rectangles $1\times k$. In conclusion, we obtain that
$$
p(n,n)
\ge
5^{m-1}
=
5^{\lfloor n/2\rfloor-2},
$$
what was to be shown. 
\end{proof}
The lower bound $p(n,n)\geq5^{\lfloor n/2\rfloor-2}$ is considerably weaker than the asymptotic formula in Theorem~\ref{asymp(n,n)} and is not asymptotically tight. For example, $p(7,7)=399618$, whereas $5^{\lfloor \frac{7}{2}\rfloor-2}=5$. Nevertheless, we retain this bound because its elementary proof gives a direct packing construction independent of the more technical proof of Theorem~\ref{asymp(n,n)}. It also provides an alternative approach for studying the asymptotic behavior of $p(n,n)$.

\section{The proof of Theorem \ref{asymp(n,n)}}
We begin with a few auxiliary lemmas. The first one gives the upper estimate for $p(n,n)$ by comparison with a variant of the colored partition function $D(\ell)$ considered in the proof of Theorem~\ref{p(n,n)stronger}.
\begin{lm}\label{lm: unrestricted rectangles}
For every positive integer $s$, we put $\rho(s)=|\{(a,b)\in\mathbb {Z}_+^2:\ a\le  b,\ ab=s\}|$, and define $R(\ell)$ by
$$
\sum_{\ell=0}^{\infty}R(\ell)x^\ell
=
\prod_{s=1}^{\infty}\tfrac{1}{(1-x^s)^{\rho(s)}}.
$$
Then $P_v(\ell)\le  R(\ell)$ for every positive integer $v$ and every non-negative integer $\ell$. Moreover,
$$
\log R(\ell)
\le 
\left(\tfrac{\pi}{\sqrt6}+o(1)\right)\sqrt{\ell\log \ell},\hspace{0.5cm}\text{as }\ell\to\infty.
$$
\end{lm}
\begin{proof}
Recall that $D(\ell)$ is defined by the generating function:
$$
\sum_{\ell=0}^{\infty}D(\ell)x^\ell
=
\prod_{s=1}^{\infty}\frac{1}{(1-x^s)^{d(s)}}.
$$

The positive integer divisors of a non-square integer $s$ form pairs $\left(a,\frac{s}{a}\right)$, where $a<\sqrt{s}<\frac{s}{a}$. Every such pair determines one rectangle type. If $s$ is a square, then the divisor $\sqrt{s}$ gives one additional rectangle type. Hence, we get that
\begin{align}\label{eq: rho}
    \rho(s)=\tfrac12\left(d(s)+\mathbf 1_{\{s\text{ is a square}\}}\right).
\end{align}

Recall that $r_v(s)$ is the number of pairs $(a,b)$ satisfying $1\le  a\le  b\le  v$ and $ab=s$. In particular, $r_v(s)=0$ for $s>v^2$. 
Since $r_v(s)\le \rho(s)$, we have that
$$
\sum_{\ell=0}^{\infty}R(\ell)x^\ell
=
\prod_{s=1}^{\infty}\tfrac{1}{(1-x^s)^{\rho(s)-r_v(s)}}
\sum_{\ell=0}^{\infty}P_v(\ell)x^\ell,
$$
where $P_v(\ell)$ is defined in Preliminaries (see also Equation \eqref{def: generating P_v(n)}). The first product on the right-hand side has constant term $1$ and all its coefficients are non-negative. Therefore, the coefficient of $x^\ell$ on the right-hand side is at least $P_v(\ell)$. 
Hence, $P_v(\ell)\le  R(\ell)$.\\
Let us define $Q(\ell)$ by 
$$\sum_{\ell=0}^{\infty}Q(\ell)x^\ell=\prod_{n=1}^{\infty}\frac{1}{1-x^{n^2}}.$$
The function $Q(\ell)$ counts partitions of $\ell$ into squares. Equality \eqref{eq: rho} gives
\begin{equation}\label{eq:RDQ-factorization}
\begin{aligned}
\left(\sum_{\ell=0}^{\infty}R(\ell)x^\ell\right)^2
&=
\prod_{s=1}^{\infty}\frac{1}{(1-x^s)^{2\rho(s)}}=
\prod_{s=1}^{\infty}\frac{1}{(1-x^s)^{d(s)}}
\prod_{n=1}^{\infty}\frac{1}{1-x^{n^2}}\\
&=
\left(\sum_{\ell=0}^{\infty}D(\ell)x^\ell\right)
\left(\sum_{\ell=0}^{\infty}Q(\ell)x^\ell\right).
\end{aligned}
\end{equation}
It is easy to observe that the sequences $D(\ell)$ and $Q(\ell)$ are weakly increasing. 
Moreover, $Q(\ell)\le  p(\ell)$. 
Equality of the coefficients of $x^{2\ell}$ in \eqref{eq:RDQ-factorization} asserts that
\begin{equation}\label{RDQrec}
\sum_{j=0}^{2\ell}R(j)R(2\ell-j)
=
\sum_{j=0}^{2\ell}D(j)Q(2\ell-j).
\end{equation}
The term corresponding to $j=\ell$ on the left-hand side is $R(\ell)^2$. Since $D(j)\le  D(2\ell)$ and $Q(2\ell-j)\le p(2\ell-j)\le p(2\ell)$, we obtain that
$$
R(\ell)^2
\le 
\sum_{j=0}^{2\ell}D(j)Q(2\ell-j)
\le 
(2\ell+1)D(2\ell)p(2\ell).
$$
Hence, $2\log R(\ell)\le \log(2\ell+1)+\log D(2\ell)+\log p(2\ell)$. The asymptotic formula for $D(\ell)$ due to Berndt et al. (see \cite[Equations (4.13) and~(4.14)]{BerndtRoblesZaharescuZeindler} and Remark~\ref{blad}) maintains that 
$$\log D(2\ell)=\left(\pi\sqrt{\tfrac{2}{3}}+o(1)\right)\sqrt{\ell\log \ell}.$$
Furthermore, the Hardy--Ramanujan formula gives $$\log p(2\ell)=O(\sqrt \ell)=o(\sqrt{\ell\log \ell}).$$ Since $\log(2\ell+1)=o(\sqrt{\ell\log \ell})$, we conclude that
$$
\log R(\ell)
\le 
\left(\tfrac{\pi}{\sqrt6}+o(1)\right)\sqrt{\ell\log \ell}.
$$
\end{proof}

\begin{re}
The function $\rho(s)$ is the number of unordered factor pairs of $s$ and corresponds to Sequence A038548 in the OEIS \cite{OEISA038548}. The function $Q(\ell)$ is the number of partitions of $\ell$ into squares and corresponds to Sequence A001156 in the OEIS \cite{OEISA001156}.
\end{re}

\begin{re}
There is also a direct proof of the upper bound in Lemma~\ref{lm: unrestricted rectangles}. Here, we only outline the main steps, omitting the details.
Since
$\rho(s)=\tfrac12\bigl(d(s)+\mathbf 1_{\{s\ \mathrm{a\ square}\}}\bigr)$, we have that
$
A(x):=\sum_{s\le  x}\rho(s)=\tfrac12x\log x+O(x).
$
By Abel summation with $f(n)=e^{-un}$ we get
$$
B(u):=\sum_{s=1}^\infty\rho(s)e^{-us}=\tfrac{1}{2u}\log\tfrac{1}{u}+O\!\left(\tfrac{1}{u}\right)
\quad(0<u\le 1).$$
We also have 
$B(u)=O(e^{-u})$ if $u\ge 1$.
Furthermore, Maclaurin series of the natural logarithm takes the form
$$\log\tfrac{1}{1-x}=\sum_{m=1}^\infty\tfrac{x^{m}}{m}$$ 
for $|x|<1$. 
Therefore, we can deduce that
\begin{align*}
    \log \sum_{k=0}^\infty R(k)e^{-tk}&=\log \prod_{s=1}^{\infty}\tfrac{1}{(1-e^{-ts})^{\rho(s)}}=\sum_{s=1}^\infty\rho(s)
\sum_{m=1}^\infty\tfrac{e^{-mts}}m\\
&=\sum_{m=1}^\infty\tfrac{1}{m}
\sum_{s=1}^\infty\rho(s)e^{-mts}=\sum_{m=1}^\infty\tfrac{B(mt)}{m}.
\end{align*} 
It can be shown that
$$
\log \sum_{k\geq0}R(k)e^{-tk}=\sum_{m\ge1}\tfrac{B(mt)}{m} \le\tfrac{\pi^{2}}{12t}\log\tfrac1t+O\!\left(\tfrac1t\right),\quad \text{as }t\to0^{+}.
$$
Since all $R(k)$ are non-negative, $R(\ell)e^{-t\ell}\le \sum_{k\geq0}R(k)e^{-tk}$. Therefore,
$$\log R(\ell)\le  t\ell+\tfrac{\pi^{2}}{12t}\log\tfrac1t+O(t^{-1}).$$ 
If we set
$t=\tfrac{\pi}{\sqrt{24}}\sqrt{\tfrac{\log\ell}{\ell}}$, then we get
$$
\log R(\ell)\le \left(\tfrac{\pi}{\sqrt6}+o(1)\right)\sqrt{\ell\log\ell}.
$$
\end{re}

The next lemma studies partitions with restricted parts. We show that the typical parts of a partition of t do not exceed $\lceil 2\sqrt{t}\log t\rceil$. We also study partitions with no small parts.
\begin{lm}\label{lm: restricted partitions}
For a positive integer $t$, we put $\mu(t)=\lceil2\sqrt t\log t\rceil$ and define $p_{\le  \mu(t)}(t)$ as the number of partitions of $t$ whose largest part is at most $\mu(t)$. Then $$p_{\le  \mu(t)}(t)\sim p(t),\ t\to\infty.$$
Moreover, for an integer $a$ satisfying $2\le  a\le  \mu(t)$, let $\mathcal A(a,t)$ be the set of all partitions $\lambda$ such that
the sum of its parts satisfies $|\lambda|\le  t$ and every part of $\lambda$ belongs to the set $\{a,a+1,\ldots,\mu(t)\}$. For all sufficiently large $t$, we have that
$$
|\mathcal A(a,t)|
\geq
\frac{p(t)}{2(t+1)^{a-1}}.
$$
\end{lm}
\begin{proof}
For every partition counted by $p(t)-p_{\le \mu(t)}(t)$, let $k$ be its largest part. If we remove one occurrence of $k$, then as a result we get a partition of $t-k$. Note that this process can be reversed. Hence, $$0\le p(t)-p_{\le  \mu(t)}(t)\le\sum_{k=\mu(t)+1}^{t}p(t-k)\le  t p(t-\mu(t)),$$
as $p$ is non-decreasing.\\
Since
$
\mu(t)=\left\lceil2\sqrt t\log t\right\rceil=o(t),
$
we have
$$
\sqrt t-\sqrt{t-\mu(t)}
=
\tfrac{\mu(t)}{\sqrt t+\sqrt{t-\mu(t)}}
=
\tfrac{2\sqrt t\log t+O(1)}
{2\sqrt t(1+o(1))}
=
(1+o(1))\log t.
$$
Hence, by the Hardy--Ramanujan formula,
$$
\tfrac{t\,p(t-\mu(t))}{p(t)}
=t\tfrac{t}{t-\mu(t)}
\exp\left(-\pi\sqrt{\tfrac{2}{3}}\left(\sqrt t-\sqrt{t-\mu(t)}\right)\right)(1+o(1))=
t^{1-\pi\sqrt{\tfrac23}+o(1)}
=
o(1),
$$
since
$
\pi\sqrt{\tfrac23}>1.
$
Therefore,
$
1-\frac{p_{\le\mu(t)}(t)}{p(t)}
=o(1),
$
that is,
$$
p_{\le\mu(t)}(t)
=(1-o(1))p(t).
$$
In particular,
$
p_{\le\mu(t)}(t)\ge\frac12p(t)
$
for all sufficiently large $t$.
For the second part of the statement, delete all parts smaller than $a$ from a partition counted by $p_{\le  \mu(t)}(t)$. The resulting partition belongs to $\mathcal A(a,t)$. 
Once the remaining partition is fixed, the original partition is determined by the multiplicities of the deleted parts $1,2,\ldots,a-1$. The multiplicity of the part $j$ is at most $\lfloor \tfrac{t}{j}\rfloor$. 
Therefore, for each fixed partition in $\mathcal A(a,t)$, there are at most
$$
\prod_{j=1}^{a-1}\left(\left\lfloor\tfrac{t}{j}\right\rfloor+1\right)
\le 
(t+1)^{a-1}
$$
partitions counted by $p_{\le  \mu(t)}(t)$ which reduce to it after deleting all parts smaller than $a$. Hence,
$$
|\mathcal A(a,t)|
\geq
\frac{p_{\le  \mu(t)}(t)}{(t+1)^{a-1}}
\geq
\frac{p(t)}{2(t+1)^{a-1}}
$$
for all sufficiently large $t$.
\end{proof}

The third lemma is the packing argument used in the construction.

\begin{lm}\label{lm: packing strips}
Let $n,q,\mu$ be positive integers satisfying $1\le  \mu<n$. Suppose that a finite multiset of integer-sided rectangles has a common height, every rectangle has width at most $\mu$, and the total width of all rectangles is at most $q(n-\mu)$. Then the rectangles can be packed into $q$ horizontal strips of width $n$ and of the same height as the rectangles.
\end{lm}
\begin{proof}
Consider the rectangles in any order. Starting from the bottom-left corner, place the rectangles from left to right in the current strip until the next rectangle does not fit. Then continue in the strip above.\\
Suppose that the next rectangle has width $b\le \mu$ and cannot be placed in the current strip. Let $w$ be the total width of the rectangles already placed in this strip. 
Then $w+b>n$, thus $w>n-b\geq n-\mu$. 

If more than $q$ strips were required, then the width used in each of the first $q$ strips would be greater than $n-\mu$. The total width used in these strips would then be greater than $q(n-\mu)$, which contradicts the assumption.
\end{proof}

The final lemma collects the estimates for the parameters used in the lower-bound construction.

\begin{lm}\label{lm: lower bound parameters}
For sufficiently large $n$, put $u_n=\lceil\log n\rceil$, $m=\left\lfloor\tfrac{\sqrt n}{\log n}\right\rfloor$, $L_n=\sum_{a=u_n}^{m}\tfrac1a$ and $\delta_n=L_n^{-\tfrac14}$. For every $a\in\{u_n,u_n+1,\ldots,m\}$, put $q_a=\left\lfloor\tfrac{n}{L_na^2}\right\rfloor$, $t_a=\left\lfloor(1-\delta_n)q_an\right\rfloor$ and $\mu_a=\left\lceil2\sqrt{t_a}\log t_a\right\rceil$. Then, we have that
\begin{align}
    &\sum_{a=u_n}^{m}a q_a\le  n;\label{ineq:reserved-height} \\
    &q_a=(1+o(1))\tfrac{n}{L_na^2},\
t_a=(1+o(1))\tfrac{n^2}{L_na^2},\quad\text{uniformly for $u_n\le  a\le  m$};\label{eq: uniform}\\
&a\le  \mu_a<\delta_n n<n
\quad\text{and}\quad
t_a\le  q_a(n-\mu_a)\quad\text{for all sufficiently large }n;\label{ineq:packing-conditions}\\
&\sum_{a=u_n}^{m}\sqrt{t_a}=\left(\tfrac1{\sqrt2}+o(1)\right)n\sqrt{\log n};\label{eq:sqrt-ta-sum}\\
&\sum_{a=u_n}^{m}(a-1)\log(t_a+1)=o(n\sqrt{\log n}).\label{eq:small-parts-error}
\end{align}
All asymptotic statements hold as $n\to \infty.$

\end{lm}
\begin{proof}
Throughout the proof, we assume that $n$ is sufficiently large
(note that all the parameters are well defined for $n\ge 133924$).\\
The estimate for harmonic sums gives
$$
L_n
=
\log\left(\tfrac{m}{u_n}\right)
+
O\left(\tfrac1{u_n}\right)
=
\tfrac12\log n
-
2\log\log n
+
O(1)
=
\left(\tfrac12+o(1)\right)\log n.
$$
Since $a q_a\le \tfrac{n}{L_na}$, we have
$$
\sum_{a=u_n}^{m}a q_a
\le 
\tfrac{n}{L_n}\sum_{a=u_n}^{m}\tfrac1a
=
n.
$$
Hence, \eqref{ineq:reserved-height} follows.\\
Furthermore, $\tfrac{n}{L_na^2}\ge \tfrac{n}{L_nm^2}\sim 2\log n$. 
Since
$
\min\limits_{u_n\le  a\le m}\frac{n}{L_na^2}
=
\frac{n}{L_nm^2}
\sim
2\log n\to\infty,
$
we obtain
$$
q_a=\left\lfloor\tfrac{n}{L_na^2}\right\rfloor
=
\left(1+o(1)\right)\tfrac{n}{L_na^2}
$$
uniformly for $u_n\le  a\le  m$, i.e.,
$
\max\limits_{u_n\le a\le m}
\left|
1-\frac{q_a}{\,\frac{n}{L_na^2}\,}
\right|
\le
\frac{L_nm^2}{n}
\sim
\frac{1}{2\log n}
\longrightarrow0.
$\\
Since
$
0\le (1-\delta_n)q_an-t_a<1,
$
we have
$$
\max\limits_{u_n\le  a\le  m}
\left|
1-\tfrac{t_a}{q_an}
\right|
\le \max\limits_{u_n\le  a\le m}\left(\delta_n+\tfrac1{q_an}\right)\le
\delta_n+\tfrac{1}{q_mn}
\sim
\left(\tfrac{2}{\log n}\right)^{\tfrac{1}{4}}
 \longrightarrow 0.
$$
Since $0\le  t_a\le  q_an$ and
$
1-\frac{t_a}{\frac{n^2}{L_na^2}}
=
\left(1-\frac{t_a}{q_an}\right)
+
\frac{t_a}{q_an}
\left(
1-\frac{q_a}{\frac{n}{L_na^2}}
\right),
$
we obtain
\begin{align*}
    \max_{u_n\le  a\le  m}
\left|
1-\frac{t_a}{\frac{n^2}{L_na^2}}
\right|
&\le 
\max_{u_n\le  a\le  m}
\left|
1-\frac{t_a}{q_an}
\right|
+
\max_{u_n\le  a\le  m}
\left|
1-\frac{q_a}{\frac{n}{L_na^2}}
\right|\\
&\le
\delta_n+\tfrac{1}{q_mn}
+\tfrac{L_nm^2}{n}
\sim
\left(\tfrac{2}{\log n}\right)^{\frac{1}{4}}
\longrightarrow0.
\end{align*}
Therefore,
$$
t_a=(1+o(1))\tfrac{n^2}{L_na^2}
$$
uniformly for $u_n\le  a\le m$.\\
The sequences $q_a$ and $t_a$ are weakly decreasing as functions of $a$, thus
$$
\min_{u_n\le  a\le  m}t_a
=
t_m
=
(1+o(1))\tfrac{n^2}{L_nm^2}
\sim
2n\log n.
$$
Note that $t_a\to\infty$ uniformly for $u_n\le  a\le  m$ as $n\to\infty$.
The above discussion, in~particular, implies~\eqref{eq: uniform}.\\
Since $t_m\le  t_a$ for every $u_n\le  a\le  m$,
$$
\max\limits_{u_n\le  a\le  m}
\frac{a}{\mu_a}
\le 
\max\limits_{u_n\le  a\le  m}
\frac{a}{2\sqrt{t_a}\log t_a}
=
\frac{m}{2\sqrt{t_m}\log t_m}
\sim
\frac{1}{2\sqrt2\,(\log n)^{\frac{5}{2}}}
\longrightarrow0
$$
Hence, $\mu_a\geq a$ for all sufficiently large $n$.\\
Since
$t_a\le  n^2, \log t_a\le 2\log n,$
and
$\sqrt{t_a}\le\sqrt{q_an}\le\frac{n}{a\sqrt{L_n}},$
we have

\begin{align*}
    \max\limits_{u_n\le  a\le  m}
\frac{\mu_a}{\delta_n n}
&\le 
\max\limits_{u_n\le  a\le  m}
\left(
\frac{2\sqrt{t_a}\log t_a}{\delta_n n}
+\frac{1}{\delta_n n}
\right)
\le 
\max\limits_{u_n\le  a\le  m}
\frac{4\log n}{\delta_n a\sqrt{L_n}}
+\frac{1}{\delta_n n}\\
&=
\frac{4\log n}{\delta_n u_n\sqrt{L_n}}
+\frac{1}{\delta_n n}
=
\frac{4\log n}{u_nL_n^{\frac{1}{4}}}
+\frac{L_n^{\frac{1}{4}}}{n}
\sim
4\sqrt[4]{\frac{2}{\log n}}
\longrightarrow0.
\end{align*}
Therefore, $\mu_a< \delta_n n<n$ for all sufficiently large $n$.
The inequalities $\mu_a<\delta_n n$ and $t_a\le (1-\delta_n)q_an$ assert that
$$
t_a\le  q_a(n-\mu_a).
$$
This concludes the proof of \eqref{ineq:packing-conditions}.\\
From
$
t_a=(1+o(1))\frac{n^2}{L_na^2}
$
uniformly for $u_n\le  a\le  m$, we obtain
$
\sqrt{t_a}
=
(1+o(1))\frac{n}{a\sqrt{L_n}}
$
uniformly for $u_n\le  a\le  m$. Since the same term $o(1)$ is valid for all $u_n\le  a\le  m$, it can be taken outside the sum.
$$
\sum_{a=u_n}^{m}\sqrt{t_a}
=
(1+o(1))
\tfrac{n}{\sqrt{L_n}}
\sum_{a=u_n}^{m}\tfrac1a
=
\left(\tfrac1{\sqrt2}+o(1)\right)n\sqrt{\log n},
$$
and we get \eqref{eq:sqrt-ta-sum}.\\
Finally, since $\log(t_a+1)\le\log(n^2+1)\le 2\log n+1$
for every $u_n\le a\le  m$, we have
$$
\sum_{a=u_n}^{m}(a-1)\log(t_a+1)
=
O\left(
\log n
\sum_{a=u_n}^{m}a
\right)
=
O(m^2\log n)
=
O\left(\tfrac{n}{\log n}\right)
=
o(n\sqrt{\log n}),
$$
which proves \eqref{eq:small-parts-error} and completes the proof.
\end{proof}
It is also worth noting that the above proof implies the following.

\begin{cor}\label{cor: min}
With the same notion as in Lemma \ref{lm: lower bound parameters} we have that $\min\limits_{u_n\le  a\le  m}t_a\to\infty$, as $n\to \infty$.
\end{cor}
Finally, we turn to the main part of this section.

\begin{proof}[Proof of Theorem \ref{asymp(n,n)}]
\emph{Upper bound.}
Lemma \ref{lm: unrestricted rectangles}, applied with $v=n$ and $\ell=n^2$, gives $p(n,n)\le  P_n(n^2)\le  R(n^2)$. Hence,
$$
\log p(n,n)
\le 
\left(\tfrac{\pi}{\sqrt6}+o(1)\right)\sqrt{n^2\log(n^2)}
=
\left(\tfrac{\pi}{\sqrt3}+o(1)\right)n\sqrt{\log n}.
$$
\emph{Lower bound.}
Let $u_n$, $m$, $L_n$, $q_a$, $t_a$ and $\mu_a$ be defined as in Lemma \ref{lm: lower bound parameters}.\\ For every $a\in\{u_n,u_{n}+1,\ldots,m\}$, suppose that there are $q_a$ horizontal strips of height $a$ and width~$n$. 
At this stage, we only reserve these strips, and their order is irrelevant. Each strip has width $n$. Later, every strip of height $a$ will be filled with rectangles of height $a$.
By \eqref{ineq:reserved-height}, all these strips can be stacked inside the square $n\times n$.\\
Let us put $\mathcal A_a=\mathcal A(a,t_a)$. Since $\min\limits_{u_n\le  a\le  m}t_a\to\infty$ (by Corollary \ref{cor: min}) and $2\le  a\le  \mu_a$, Lemma \ref{lm: restricted partitions} gives
$$
|\mathcal A_a|
\geq
\frac{p(t_a)}{2(t_a+1)^{a-1}}
$$
simultaneously for all $a\in\{u_n,u_n+1,\ldots,m\}$. Choose independently a partition $\lambda_a\in\mathcal A_a$ for every such $a$, and replace every part $b$ of $\lambda_a$ by a rectangle $a\times b$. 
These rectangles have common height $a$, their widths are at most $\mu_a$, and, by \eqref{ineq:packing-conditions}, their total width satisfies $|\lambda_a|\le  t_a\le  q_a(n-\mu_a)$. Lemma \ref{lm: packing strips} shows that they can be packed into the $q_a$ reserved strips of height $a$.\\
Fill the unused part of these strips with $q_an-|\lambda_a|$ rectangles $a\times1$. Put $h_n=n-\sum_{a=u_n}^{m}a q_a$. Then $h_n\in\mathbb Z_{\geq0}$, and the remaining strip of width $n$ and height $h_n$ can be filled with $h_n$ rectangles $1\times n$. The resulting multiset tiles the square $n\times n$.\\
The map from the family $(\lambda_a)_{a=u_n}^{m}$ to the resulting rectangular partition is injective. Indeed, every selected rectangle has both side lengths at least $2$, whereas every added rectangle has a side of length $1$. Deleting all rectangles having a side of length $1$ recovers the selected rectangles. Every selected rectangle has the canonical form $a\times b$ with $a\le b$, thus its smaller side determines $a$, while its larger side recovers the corresponding part of $\lambda_a$. Therefore, identifying rectangles up to rotation does not make two distinct selected rectangles equal, and the resulting multiset uniquely determines every partition $\lambda_a$. Hence,
$$
p(n,n)
\geq
\prod_{a=u_n}^{m}|\mathcal A_a|.
$$
It follows that
$$
\log p(n,n)
\geq
\sum_{a=u_n}^{m}\log p(t_a)
-
\sum_{a=u_n}^{m}(a-1)\log(t_a+1)
-
(m-u_n+1)\log2.
$$
Since
$
\min\limits_{u_n\le  a\le  m} t_a\to\infty,
$
the Hardy--Ramanujan formula holds uniformly for \mbox{$u_n\le  a\le  m$}. The uniformity is with respect to $a$ as $n\to\infty$. Therefore,
$$
\sum_{a=u_n}^{m}
\left(\pi\sqrt{\tfrac23}+o(1)\right)\sqrt{t_a}
=
\left(\pi\sqrt{\tfrac23}+o(1)\right)
\sum_{a=u_n}^{m}\sqrt{t_a}.
$$
Equations \eqref{eq:sqrt-ta-sum} and \eqref{eq:small-parts-error} give the required estimates, i.e. 
$$
\sum_{a=u_n}^{m}\sqrt{t_a}
=
\left(\tfrac{1}{\sqrt2}+o(1)\right)n\sqrt{\log n},\quad\text{and}\quad \sum_{a=u_n}^{m}(a-1)\log(t_a+1)=o(n\sqrt{\log n}).$$ 
Therefore,
\begin{align*}
\sum_{a=u_n}^{m}\log p(t_a)
&=
\left(\pi\sqrt{\tfrac23}+o(1)\right)
\sum_{a=u_n}^{m}\sqrt{t_a}=
\left(\pi\sqrt{\tfrac23}+o(1)\right)
\left(\tfrac1{\sqrt2}+o(1)\right)
n\sqrt{\log n}
\\&=
\left(\tfrac{\pi}{\sqrt3}+o(1)\right)
n\sqrt{\log n}.
\end{align*}
Moreover, $m-u_n+1=O(\tfrac{\sqrt n}{\log n})=o(n\sqrt{\log n})$. We obtain
$$
\log p(n,n)
\geq
\left(\tfrac{\pi}{\sqrt3}-o(1)\right)n\sqrt{\log n}.
$$
Together with the upper bound, this gives
$$
p(n,n)
=
\exp\left(
\left(\tfrac{\pi}{\sqrt3}+o(1)\right)n\sqrt{\log n}
\right).
$$

\end{proof}
\begin{re}
The proof of Theorem~\ref{asymp(n,n)} also determines the logarithmic asymptotic behavior of $P_n(n^2)$. Indeed, the inequalities $p(n,n)\le  P_n(n^2)\le  R(n^2)$ and the lower and upper estimates established in the proof give
$
\log P_n(n^2)
\sim
\log p(n,n)
\sim
\tfrac{\pi}{\sqrt3}n\sqrt{\log n}.
$
\end{re}

\begin{re}
The preceding arguments strongly suggest the following estimates for nonsquare rectangles. Put $h=\min\{m,n\}$ and $w=\max\{m,n\}$. We conjecture that, if $h\to\infty$, then
$$
\exp\left(
\left(\tfrac{\pi}{\sqrt3}-o(1)\right)\sqrt{hw\log h}
\right)
\le 
p(h,w)
\le 
\exp\left(
\left(\tfrac{\pi}{\sqrt6}+o(1)\right)\sqrt{hw\log(hw)}
\right).
$$
In particular, if $\log w\sim\log h$, we conjecture that
$$
p(h,w)
=
\exp\left(
\left(\tfrac{\pi}{\sqrt3}+o(1)\right)\sqrt{hw\log h}
\right).
$$
We also conjecture that, if $h$ is fixed and $w\to\infty$, then
$$
\exp\left(
\left(\pi\sqrt{\tfrac{2h}{3}}-o(1)\right)\sqrt w
\right)
\le 
p(h,w)
\le 
\exp\left(
\left(\pi\sqrt{\tfrac{2hH_h}{3}}+o(1)\right)\sqrt w
\right),
$$
where $H_h=\sum_{j=1}^{h}\tfrac1j$.
\end{re}

\section{The proof of Theorem \ref{thmp(3n)UB}}

In this section we derive the upper bound for the number of partitions of the rectangle $3\times n$. For the first few values of $p(3,n)$, we refer the reader to Sequence A360632 in the OEIS \cite{OEISA360632}. 

At first, let us define $a_3(n)$ as the number of ordinary partitions of $n$ into parts at least $3$. In particular $a_3(0)=1$ and $a_3(1)=a_3(2)=0$. The generating function of $a_3(n)$ takes the form
\begin{equation}\label{eq:a3-generating-function}
\sum_{n=0}^\infty a_3(n)x^n
=
\prod_{n=3}^\infty\frac{1}{1-x^n}
=
(1-x)(1-x^2)\prod_{n=1}^\infty\frac{1}{1-x^n},
\end{equation}
which implies that
$$
a_3(n)=p(n)-p(n-1)-p(n-2)+p(n-3)
$$
with $p(j)=0$ for $j<0$. The above equality might be also proved combinatorially using the inclusion-exclusion principle.

To prove Theorem \ref{thmp(3n)UB}, we need a few auxiliary results.

Let $\widehat p(3,m)$ denote the number of partitions of the rectangle $3\times m$ that do not contain blocks of type $3\times j$ with $j\ge 3$. 
We use the convention $p(3,0)=\widehat p(3,0)=1$, corresponding to the empty partition. Every partition of the $3 \times n$ rectangle can be rearranged in such a way that all blocks of type $3 \times k$ for $k \geq 3$ are shifted to the left side of the rectangle. It means that, for any partition of $3 \times n$, there exists a unique $i \in\{ 0,1, \ldots, n\}$ such that the first $i$ columns of the rectangle gives a~partition of $3 \times i$ using only blocks of type $3 \times k$ for $k\geq3$, and the remaining $n-i$ columns of the rectangle gives a~partition of $3 \times (n-i)$  using only blocks enumerated by $\widehat p(3,n-i)$. In other words, it follows that
\begin{equation}\label{eq:p3-decomposition}
p(3,n)=\sum_{i=0}^{n}a_3(i)\widehat p(3,n-i).
\end{equation}
We now estimate $\widehat p(3,n)$ using only the total area argument presented in Section 2 and Section 3. The possible block types are $1\times i$ and $2\times j$ for $i,j\geq1$, together with $3\times1$ and $3\times2$. To obtain an upper bound, all the mentioned block types are considered to be different. 

Further, let $S(n)$ denote the set of all multisets of these block types whose total area is equal to $3n$. The multisets do not need to satisfy condition $(3)$ of Definition \ref{def: rectangle partition}. Observe that every partition counted by $\widehat p(3,n)$ determines an element of $S(n)$. Hence, 
\begin{equation}\label{p3n<=Sn}
\widehat p(3,n)\le  |S(n)|.
\end{equation}


Combining \eqref{eq:p3-decomposition} with \eqref{p3n<=Sn} gives the following lemma.
\begin{lm}\label{bcbx}
For every $n\geq 0$,
$$
p(3,n)\le  \sum_{i=0}^{n}a_3(i)|S(n-i)|.
$$
\end{lm}

Our next goal is to derive the asymptotic behavior of $|S(n)|$.

\begin{lm}\label{Gcomasym}
Let us fix $\Delta>0$ and put
$$
G(x)=
\tfrac{1}{(1-x^3)(1-x^6)}
\left(\prod_{j=1}^\infty\tfrac{1}{1-x^j}\right)
\left(\prod_{j=1}^\infty\tfrac{1}{1-x^{2j}}\right).
$$
As the complex variable $z\to0$ in the region
$
\operatorname{Re}z>0
$
and
$
|\operatorname{Im}z|
\le 
\Delta\operatorname{Re}z,
$
we have that
$$
G(e^{-z})
\sim
\tfrac{1}{18\pi\sqrt2}\,z^{-1}
\exp\left(\tfrac{\pi^2}{4z}\right),
$$
uniformly in that region.
\end{lm}

\begin{proof}
We apply the transformation formula for the Dedekind eta function (see \cite[Chapter~3]{ApostolModularFunctions}):
$$
\eta\left(-\tfrac1\tau\right)
=
\sqrt{-i\tau}\,\eta(\tau).
$$
Let us set $\tau=\frac{iz}{2\pi}.$
Then
$\eta\left(\frac{2\pi i}{z}\right)
=
\sqrt{\frac{z}{2\pi}}\,
\eta\left(\frac{iz}{2\pi}\right),
$
where the principal branch of the square root is used. We also have that
$$
\eta\left(\frac{iz}{2\pi}\right)
=
e^{-z/24}
\prod_{n=1}^\infty\left(1-e^{-nz}\right).
$$
Since
$$
\operatorname{Re}\!\left(\tfrac{1}{z}\right)
=
\tfrac{\operatorname{Re}z}{|z|^2}
\ge
\tfrac{1}{(1+\Delta^2)\operatorname{Re}z},
$$
it follows that
$$
\left|
e^{-4\pi^2\tfrac{n}{z}}
\right|
=
e^{-4\pi^2n\operatorname{Re}\left(\tfrac{1}{z}\right)}
\le
e^{-\frac{4\pi^2n}{(1+\Delta^2)\operatorname{Re}z}},
$$
which tends to $0$ uniformly as $z\to0$ in the region
$
|\operatorname{Im}z|
\le
\Delta\operatorname{Re}z.
$ We have
$
0<\exp\left(-4\pi^2\operatorname{Re}\tfrac1z\right)<1
$
and
$$
\sum_{j=1}^{\infty}
\left|
\exp\left(-\tfrac{4\pi^2j}{z}\right)
\right|
=
\tfrac{\exp\left(-4\pi^2\operatorname{Re}\tfrac1z\right)}{1-\exp\left(-4\pi^2\operatorname{Re}\tfrac1z\right)}
\longrightarrow0
$$
uniformly in the region $
|\operatorname{Im}z|
\le
\Delta\operatorname{Re}z.
$ Hence,
$$
\prod_{j=1}^{\infty}
\left(
1-\exp\left(-\tfrac{4\pi^2j}{z}\right)
\right)
\longrightarrow1
$$
uniformly in the region $
|\operatorname{Im}z|
\le
\Delta\operatorname{Re}z.
$
Thus, we get that
$$
\eta\left(\tfrac{2\pi i}{z}\right)
=
e^{-\tfrac{\pi^2}{6z}}
\prod_{n=1}^\infty\left(1-e^{-4\pi^2n/z}\right)
\sim
e^{-\tfrac{\pi^2}{6z}}
$$
uniformly as $z\to0$ in the region
$
|\operatorname{Im}z|
\le 
\Delta\operatorname{Re}z.
$
It also gives us that
$$
e^{-\tfrac{z}{24}}
\prod_{n=1}^\infty\left(1-e^{-nz}\right)
\sim
\sqrt{\tfrac{2\pi}{z}}\,
e^{-\tfrac{\pi^2}{6z}}.
$$
Since $e^{z/24}\to1$, it follows that
$$
\prod_{n=1}^\infty\left(1-e^{-nz}\right)
\sim
\sqrt{\tfrac{2\pi}{z}}\,
\exp\left(-\tfrac{\pi^2}{6z}\right).
$$
That asymptotic asserts that
\begin{align*}
    \prod_{n=1}^\infty\tfrac{1}{1-e^{-nz}}\sim\sqrt{\tfrac{z}{2\pi}}\exp\left(\tfrac{\pi^2}{6z}\right)\hspace{0.5cm}\text{and}\hspace{0.5cm}\prod_{n=1}^\infty\tfrac{1}{1-e^{-2nz}}\sim\sqrt{\tfrac{z}{\pi}}\exp\left(\tfrac{\pi^2}{12z}\right).
\end{align*}
Moreover, we also have that
\begin{align*}
    \tfrac{1}{(1-e^{-3z})(1-e^{-6z})}&\sim\tfrac{1}{18z^2}.
\end{align*}
In consequence, we conclude that
$$
G(e^{-z})
\sim
\tfrac{1}{18\pi\sqrt2}\,z^{-1}
\exp\left(\tfrac{\pi^2}{4z}\right),
$$
what was to be shown.
\end{proof}


\begin{re}
Lemma~\ref{Gcomasym} can be also verified numerically. 

\begin{table}[htb]
\renewcommand{\arraystretch}{1.4}
\begin{tabular}{|c|c|c|c|c|c|c|}\hline
\backslashbox{$c$}{$t$} 
& $0.05$ & $0.01$ & $0.005$ & $0.002$ & $0.001$ & $0.00001$\\
\hline\hline
$0$ & $0.238704$ & $0.044525$ & $0.022068$ & $0.008781$ & $0.004383$ & $0.000043751$\\
$0.25$ & $0.246062$ & $0.045896$ & $0.022747$ & $0.009051$ & $0.004518$ & $0.000045097$\\
$0.5$ & $0.266932$ & $0.049781$ & $0.024673$ & $0.009817$ & $0.004900$ & $0.000048915$\\
$1$ & $0.337847$ & $0.062970$ & $0.031209$ & $0.012418$ & $0.006198$ & $0.000061873$\\
$2$ & $0.535473$ & $0.099572$ & $0.049347$ & $0.019635$ & $0.009800$ & $0.000097830$\\
$3$ & $0.760398$ & $0.140833$ & $0.069789$ & $0.027768$ & $0.013859$ & $0.000138352$\\\hline
\end{tabular}
\medskip
\caption{The values of $\left|\frac{G(e^{-z})}{\frac{1}{18\pi\sqrt2}\,z^{-1}\exp\left(\frac{\pi^2}{4z}\right)}-1\right|$ for some small numbers $c$ and $t$, where $z=t(1+ci)$ and $c=\frac{|\operatorname{Im}z|}{\operatorname{Re}z}$.}
\end{table}
We see that for every fixed value of $c=\frac{|\operatorname{Im}z|}{\operatorname{Re}z},$ the error decreases as $t\to0^+$. This observation is consistent with Lemma~\ref{Gcomasym}.
\end{re}
We recall a form of Ingham's Tauberian theorem \cite{Ingham1941} due to Bringmann\ Jennings--Shaffer and Mahlburg; see \cite[Theorem~1.1]{BringmannJenningsShafferMahlburg2021}. The theorem will be used several times below.

\begin{thm}\label{thm: BJSM}
Suppose that $B(x)=\sum_{j=0}^\infty b_jx^j$ is a power series with non-negative real coefficients and radius of convergence at least one. Let $\lambda,\gamma>0$ and $\alpha,\beta\in\mathbb{R}$. Suppose that
\begin{align*}
B(e^{-t})
&\sim
\lambda
\left(\log\frac1t\right)^\alpha
t^\beta
\exp\left(\frac{\gamma}{t}\right),
\qquad\text{as }t\to0^+,\\
B(e^{-z})
&=
O\left(
\left(\log\frac1{|z|}\right)^\alpha
|z|^\beta
\exp\left(\frac{\gamma}{|z|}\right)
\right),
\qquad\text{as }z\to0,
\end{align*}
where $z=x+iy$, $x,y\in\mathbb{R}$, $x>0$, and the second estimate holds in every region of the~form $|y|\le \Delta x$ with $\Delta>0$. Then
$$
\sum_{k=0}^n b_k
\sim
\frac{
\lambda\gamma^{\frac{\beta}{2}-\frac{1}{4}}(\log n)^\alpha
}{
2^{\alpha+1}\sqrt{\pi}\,
n^{\frac{\beta}{2}+\frac{1}{4}}
}
\exp\left(2\sqrt{\gamma n}\right),
\qquad\text{as }n\to\infty.
$$
Furthermore, if the sequence $(b_n)_{n\geq0}$ is weakly increasing, then
$$
b_n
\sim
\frac{
\lambda\gamma^{\frac{\beta}{2}+\frac{1}{4}}(\log n)^\alpha
}{
2^{\alpha+1}\sqrt{\pi}\,
n^{\frac{\beta}{2}+\frac{3}{4}}
}
\exp\left(2\sqrt{\gamma n}\right),
\qquad\text{as }n\to\infty.
$$
\end{thm}
We first apply Theorem~\ref{thm: BJSM} to determine the asymptotic behavior of $a_3(n)$.
\begin{thm}\label{thma3}
Let $a_3(n)$ denote the number of partitions of $n$ into parts at least $3$. Then, 
$$
a_3(n)
\sim
\tfrac{\pi^2}{12\sqrt3}\,
n^{-2}
\exp\left(\pi\sqrt{\tfrac{2n}{3}}\right),\qquad\text{as }n\to\infty.
$$
\end{thm}

\begin{proof}
Let
$
A(x):=\sum_{n=0}^{\infty}a_3(n)x^n.
$
By \eqref{eq:a3-generating-function},
$
A(x)
=
(1-x)(1-x^2)
\prod_{j=1}^{\infty}\frac{1}{1-x^j}.
$\\
We use the transformation formula
$
\eta\left(-\frac1\tau\right)
=
\sqrt{-i\tau}\,\eta(\tau)
$
for the Dedekind eta function
$
\eta(\tau)
=
e^{\pi i\tau/12}
\prod_{j\geq1}\left(1-e^{2\pi i j\tau}\right)
$
(see \cite[Chapter~3]{ApostolModularFunctions}). Put
$
\tau=\frac{iz}{2\pi}.
$
Then
$
\eta\left(\frac{2\pi i}{z}\right)
=
\sqrt{\frac{z}{2\pi}}\,
\eta\left(\frac{iz}{2\pi}\right),
$
where the principal branch of the square root is used. Moreover,
$$
\eta\left(\tfrac{iz}{2\pi}\right)
=
e^{-z/24}
\prod_{j=1}^\infty\left(1-e^{-jz}\right)
$$
and
$$
\eta\left(\tfrac{2\pi i}{z}\right)
=
\exp\left(-\tfrac{\pi^2}{6z}\right)
\prod_{j=1}^\infty
\left(1-\exp\left(-\tfrac{4\pi^2j}{z}\right)\right).
$$
Let $\Delta>0$ be fixed and suppose that
$
\operatorname{Re}z>0,
\
|\operatorname{Im}z|
\le 
\Delta\operatorname{Re}z.
$ Then
$
\operatorname{Re}\frac1z
=
\frac{\operatorname{Re}z}{|z|^2}
\geq
\frac{1}{(1+\Delta^2)\operatorname{Re}z}
$
and $\lim\limits_{z\to0} \frac{1}{(1+\Delta^2)\operatorname{Re}z}
=\infty$. Hence,
$$
0\le \sum_{j=1}^\infty
\left|
\exp\left(-\frac{4\pi^2j}{z}\right)
\right|
\le
\frac{
\exp\left(-4\pi^2\operatorname{Re}\frac1z\right)
}{
1-\exp\left(-4\pi^2\operatorname{Re}\frac1z\right)
},
$$
where $ \lim\limits_{z\to0}\frac{
\exp\left(-4\pi^2\operatorname{Re}\frac1z\right)
}{
1-\exp\left(-4\pi^2\operatorname{Re}\frac1z\right)}=0.$ Therefore,
$
\lim\limits_{z\to0}\prod_{j\geq1}
\left(1-\exp\left(-\frac{4\pi^2j}{z}\right)\right)
=1
$
uniformly in the considered region. The transformation formula gives
$$
\prod_{j=1}^\infty\frac{1}{1-e^{-jz}}=\sqrt{\tfrac{z}{2\pi}}\,
\exp\left(\tfrac{\pi^2}{6z}-\tfrac{z}{24}\right)
\prod_{j=1}^\infty
\frac{1}{
1-\exp\left(-\frac{4\pi^2j}{z}\right)
}.
$$
Consequently,
\begin{equation}\label{etafor}
\prod_{j=1}^\infty\frac{1}{1-e^{-jz}}
\sim
\sqrt{\tfrac{z}{2\pi}}\,
\exp\left(\tfrac{\pi^2}{6z}\right)
\end{equation}
uniformly as $z\to0$ in the region
$
\operatorname{Re}z>0,
|\operatorname{Im}z|
\le
\Delta\operatorname{Re}z.
$ Further,
$
1-e^{-z}\sim z$ and\linebreak 
$1-e^{-2z}\sim2z.
$
Therefore,
$$
A(e^{-z})
\sim
\sqrt{\tfrac{2}{\pi}}\,
z^{\frac{5}{2}}
\exp\left(\tfrac{\pi^2}{6z}\right)
$$
uniformly in the region
$
\operatorname{Re}z>0,
|\operatorname{Im}z|
\le
\Delta\operatorname{Re}z.
$

Since $a_3(0)=1$ and $a_3(1)=0$, put
$
\widetilde A(x):=A(x)-1.
$
Observe that, for every $n\ge3$,
$
a_3(n)\le  a_3(n+1).
$ 
Hence, the coefficients of $\widetilde A(x)$ are non-negative and weakly increasing.
We also have
$
\widetilde A(e^{-z})
\sim
\sqrt{\frac{2}{\pi}}\,
z^{\frac{5}{2}}
\exp\left(\frac{\pi^2}{6z}\right)
$ uniformly as $z\to0$ in the region.

The uniform asymptotic and the inequality
$
\left|
\exp\left(\frac{\pi^2}{6z}\right)
\right|
=
\exp\left(
\frac{\pi^2}{6}\operatorname{Re}\frac1z
\right)
\le
\exp\left(\frac{\pi^2}{6|z|}\right)
$
give the complex estimate
$$
\widetilde A(e^{-z})
=
O\left(
|z|^{\frac{5}{2}}
\exp\left(\tfrac{\pi^2}{6|z|}\right)
\right).
$$
Thus, the assumptions of Theorem~\ref{thm: BJSM} are satisfied with
$
\lambda=\sqrt{\frac{2}{\pi}},
\
\alpha=0,
\
\beta=\frac52,
\
\gamma=\frac{\pi^2}{6}.
$
Since the coefficient of $x^n$ in $\widetilde A(x)$ equals $a_3(n)$ for every $n\geq1$, Theorem~\ref{thm: BJSM} gives
$
a_3(n)
\sim
\frac{\lambda}{2\sqrt\pi}\,
\gamma^{\frac{\beta}{2}+\frac{1}{4}}
n^{-\frac{\beta}{2}-\frac{3}{4}}
\exp\left(2\sqrt{\gamma n}\right).
$
Therefore,
$$
a_3(n)
\sim
\tfrac{\pi^2}{12\sqrt3}\,
n^{-2}
\exp\left(\pi\sqrt{\tfrac{2n}{3}}\right).
$$
\end{proof}

\begin{thm}\label{Gc}
Let
\begin{equation}\label{eq:G-coefficients}
G(x)=\frac{1}{(1-x^3)(1-x^6)}
\left(\prod_{n=1}^\infty\frac{1}{1-x^n}\right)
\left(\prod_{n=1}^\infty\frac{1}{1-x^{2n}}\right)
=
\sum_{n=0}^{\infty}b_nx^n.
\end{equation}
Then
$$
b_n
\sim
\tfrac{1}{36\pi^2}
n^{-\frac{1}{4}}
\exp\left(\pi\sqrt n\right),
\hspace{0.5cm}\text{as }n\to\infty.
$$
\end{thm}

\begin{proof}
We apply Theorem~\ref{thm: BJSM} to the generating function $G(x)$. First, we can write\\ $G(x)=\frac{1}{1-x}H(x),$
where
$$
H(x)=\frac{1}{(1-x^3)(1-x^6)}
\left(\prod_{n=2}^\infty\frac{1}{1-x^n}\right)
\left(\prod_{n=1}^\infty\frac{1}{1-x^{2n}}\right).
$$
The coefficients of both $G(x)$ and $H(x)$ are non-negative. Moreover, if
$H(x)=\sum_{n=0}^{\infty}h_nx^n,$
then
$b_n=\sum_{j=0}^{n}h_j.
$
In particular, this implies that the sequence $(b_n)_{n\ge0}$ is weakly increasing.

For every $\Delta>0$, Lemma~\ref{Gcomasym} guarantees that
$$
G(e^{-z})
\sim
\tfrac{1}{18\pi\sqrt2}\,
z^{-1}
\exp\left(\tfrac{\pi^2}{4z}\right)
$$
uniformly as $z\to0$ in the region
$
\operatorname{Re}z>0
$
and
$
|\operatorname{Im}z|\le \Delta\operatorname{Re}z.
$
This uniform asymptotic also gives the required complex estimate. Indeed, one has
$$
G(e^{-z})
=
O\!\left(
|z|^{-1}
\left|
\exp\left(\tfrac{\pi^2}{4z}\right)
\right|
\right)
$$
for sufficiently small $z$ in the region. Since
$
\left|
\exp\left(\tfrac{\pi^2}{4z}\right)
\right|
=
\exp\left(
\frac{\pi^2}{4}
\operatorname{Re}\frac{1}{z}
\right)
$
and
$
\operatorname{Re}\frac{1}{z}
\le 
\frac{1}{|z|},
$
we obtain that
$$
G(e^{-z})=
O\left(
|z|^{-1}
\exp\left(\tfrac{\pi^2}{4|z|}\right)\right).
$$
Therefore, the assumptions of Theorem \ref{thm: BJSM} are satisfied with
\begin{align*}
    \lambda=\tfrac{1}{18\pi\sqrt2},\hspace{0.5cm} \alpha=0,\hspace{0.5cm}\beta=-1\hspace{0.5cm}\text{and}\hspace{0.5cm}\gamma=\tfrac{\pi^2}{4}.
\end{align*}
Hence, we have that
$$
b_n
\sim
\tfrac{\lambda}{2\sqrt\pi}
\gamma^{\tfrac{\beta}{2}+\tfrac{1}{4}}
n^{-\tfrac{\beta}{2}-\tfrac{3}{4}}
\exp\left(2\sqrt{\gamma n}\right).
$$
After substituting, one can conclude that
$$
b_n
\sim
\tfrac{1}{36\pi^2}
n^{-\frac{1}{4}}
\exp\left(\pi\sqrt n\right).
$$
\end{proof}

The last auxiliary result is Murty's theorem concerning asymptotic of the convolution of two generating functions (see \cite[Theorem~3]{Murty}).

\begin{thm}\label{thm: Murty}
    Suppose that
    \begin{align*}
        f(x)=\sum_{n=0}^\infty\lambda_f(n)x^n\hspace{0.5cm}\text{and}\hspace{0.5cm}g(x)=\sum_{n=0}^\infty\lambda_g(n)x^n,
    \end{align*}
    where
    \begin{align*}
        \lambda_f\sim c_fn^\alpha e^{A\sqrt{n}},\hspace{0.5cm} \lambda_g\sim c_gn^\beta e^{B\sqrt{n}}
    \end{align*}
    with $\alpha,\ \beta,\ A,\ B,\ c_f,\ c_g\in\mathbb{R}$ and $A,\ B,\ c_f,\ c_g>0$. Then, for
    \begin{align*}
        (fg)(x)=\sum_{n=0}^\infty\lambda_{fg}(n)x^n,
    \end{align*}
    we have 
    \begin{align*}
        \lambda_{fg}(n)\sim c_fc_g2\sqrt{2\pi}\frac{A^{2\alpha+1}B^{2\beta+1}}{(A^2+B^2)^{\alpha+\beta+\frac{5}{4}}}n^{\alpha+\beta+\frac{3}{4}}\exp\left(\sqrt{(A^2+B^2)n}\right),\hspace{0.5cm}\text{as }n\to\infty.
    \end{align*}
\end{thm}

Finally, we are in position to show Theorem \ref{thmp(3n)UB}.

\begin{proof}[Proof of Theorem \ref{thmp(3n)UB}]
By \eqref{eq:G-coefficients} and the definition of $S(n)$, we have $|S(n)|=b_{3n}$.
Taking $3n$ instead of $n$ in Theorem \ref{Gc} gives us
$$
|S(n)|=b_{3n}
\sim
\tfrac{3^{-\frac{1}{4}}}{36\pi^2}\,
n^{-\frac{1}{4}}\exp\left(\pi\sqrt{3n}\right).
$$
On the other hand, Theorem \ref{thma3} ensures that
$$
a_3(n)
\sim
\tfrac{\pi^2}{12\sqrt3}\,
n^{-2}\exp\left(\pi\sqrt{\tfrac{2n}{3}}\right).
$$
We want to estimate the convolution
$$\sum_{i=0}^na_3(i)|S(n-i)|.$$
To do that, let us write
\begin{align*}
    a_3(n)\sim c_a n^\alpha\exp(A\sqrt n)\hspace{0.5cm}\text{and}\hspace{0.5cm}|S(n)|\sim c_s n^\beta\exp(B\sqrt n),
\end{align*}
where
\begin{align*}
    c_a=\tfrac{\pi^2}{12\sqrt3},\
\alpha=-2,\
A=\pi\sqrt{\tfrac{2}{3}},\hspace{0.5cm}\text{and}\hspace{0.5cm}c_s=\tfrac{3^{-\frac{1}{4}}}{36\pi^2},\
\beta=-\tfrac14,\
B=\pi\sqrt3.
\end{align*}
Both sequences are therefore of the form considered in Theorem \ref{thm: Murty}. Hence, we obtain that
$$
\sum_{i=0}^{n}a_3(i)|S(n-i)|\sim C n^{\alpha+\beta+\frac{3}{4}}\exp\left(\sqrt{(A^2+B^2)n}\right),
$$
where
$$
C=c_ac_s\,2\sqrt{2\pi}\,\frac{A^{2\alpha+1}B^{2\beta+1}}{(A^2+B^2)^{\alpha+\beta+\frac{5}{4}}}.
$$
In the present case,
$A^2+B^2=\frac{11\pi^2}{3},
$
and
$\alpha+\beta+\frac34=-\frac32$.\\
Substitution of the constants gives us 
\begin{equation}\label{eq:p3-convolution-asymptotic}
\sum_{i=0}^{n}a_3(i)|S(n-i)|
\sim\tfrac{11}{432}n^{-\frac{3}{2}}\exp\left(\pi\sqrt{\tfrac{11n}{3}} \right),\hspace{0.5cm}\text{as }n\to\infty.
\end{equation}
Finally, we know by Lemma \ref{bcbx} that $p(3,n)\le \sum_{i=0}^{n}a_3(i)|S(n-i)|$. Hence, we obtain the required upper bound for $p(3,n)$.
\end{proof}

\begin{re}
The exponential constant in Theorem~\ref{thmp(3n)UB} is $\pi\sqrt{\tfrac{11}{3}}$. Theorem~\ref{thm: asymptotic p(3,n)} proves that this constant is exact. 
\end{re}

\section{The proof of Theorem \ref{thm: asymptotic p(3,n)}}

We begin with four auxiliary lemmas.

\begin{lm}\label{lm: unrestricted p3}
For a non-negative integer $\ell$, let $U_3(\ell)$ denote the number of multisets of rectangle types $a\times b$, where $1\le  a\le 3$ and $a\le  b$, whose total area is $\ell$. The multiset need not form a tiling. Then
$$
p(3,n)\le  U_3(3n)
$$
and
$$
U_3(3n)
\sim
\tfrac{121\sqrt{33}\pi^3}{23328}\,
n^{-3}
\exp\left(
\pi\sqrt{\tfrac{11n}{3}}
\right).
$$
\end{lm}

\begin{proof}
Put
$
F_3(x):=\sum_{\ell=0}^{\infty}U_3(\ell)x^\ell.
$
For a fixed value of $a\in\{1,2,3\}$, the permitted rectangle types are $a\times b$ with $b\geq a$. Hence,
$$
F_3(x)
=
\prod_{b=1}^{\infty}\frac{1}{1-x^b}
\prod_{b=2}^{\infty}\frac{1}{1-x^{2b}}
\prod_{b=3}^{\infty}\frac{1}{1-x^{3b}}.
$$
If
$
P(x):=\prod_{n=1}^{\infty}(1-x^n)^{-1},
$
which converges for $|x|<1$, then
$$
F_3(x)
=
(1-x^2)(1-x^3)(1-x^6)
P(x)P(x^2)P(x^3),
$$
and this product converges for $|x|<1$ as well, in particular $F_3$ has radius of convergence at least $1$.
Notice that every partition of the rectangle $3\times n$ determines a multiset counted by $U_3(3n)$. Therefore, $0\le  p(3,n)\le U_3(3n)$.\\ 
Let $\Delta>0$ and a positive integer $k$ be fixed. If $\operatorname{Re}z>0$ and $|\operatorname{Im}z|\le \Delta\operatorname{Re}z$, then $\operatorname{Re}(kz)>0$ and
$
|\operatorname{Im}(kz)|
=
k|\operatorname{Im}z|
\le 
\Delta k\operatorname{Re}z
=
\Delta\operatorname{Re}(kz).
$
Hence, replacing $z$ by $kz$ in \eqref{etafor}, we obtain
$$
P(e^{-kz})
=
\prod_{j=1}^{\infty}\tfrac{1}{1-e^{-kzj}}
\sim
\sqrt{\tfrac{kz}{2\pi}}\,
\exp\left(
\tfrac{\pi^2}{6kz}
\right)
$$
uniformly as $z\to0$ with $\operatorname{Re}z>0$ and $|\operatorname{Im}z|\le \Delta\operatorname{Re}z$.
Moreover,
$$
(1-e^{-2z})(1-e^{-3z})(1-e^{-6z})\sim36z^3, 
$$
uniformly as $z\to 0.$\\
Thus, we have
\begin{equation}\label{assF3}
F_3(e^{-z})
\sim
\tfrac{18\sqrt3}{\pi^{\frac{3}{2}}}\,
z^{\frac{9}{2}}
\exp\left(
\tfrac{11\pi^2}{36z}
\right)
\end{equation}
uniformly as $z\to0$ with $\operatorname{Re}z>0$ and $|\operatorname{Im}z|\le \Delta\operatorname{Re}z$. Here $z^{\frac{1}{2}}$ denotes the principal branch, which is well defined since $|\arg z|<\tfrac{\pi}{2}$ in this region.\\
In particular, taking $z=t>0$ real, we obtain
$$
F_3(e^{-t})
\sim
\tfrac{18\sqrt3}{\pi^{\frac{3}{2}}}\,
t^{\frac{9}{2}}
\exp\left(
\tfrac{11\pi^2}{36t}
\right),
\ \ \ \text{as }t\to0^+.
$$
By the uniform asymptotic in \eqref{assF3} and the inequality
$
|\exp(\tfrac{11\pi^2}{36z})|
\le
\exp(\tfrac{11\pi^2}{36|z|})
$
we get
$$
F_3(e^{-z})
=
O\left(
|z|^{\frac{9}{2}}
\exp\left(
\tfrac{11\pi^2}{36|z|}
\right)
\right),
$$
uniformly as $z\to 0$ with $\operatorname{Re}z>0$ and
$|\operatorname{Im}z|\le \Delta\operatorname{Re}z$.\\
The sequence $U_3(\ell)$ is weakly increasing, since adding one rectangle of type $1\times1$ defines an injection from the multisets counted by $U_3(\ell)$ into those counted by $U_3(\ell+1)$. Therefore, one can apply Theorem~\ref{thm: BJSM}  with $\lambda=\tfrac{18\sqrt3}{\pi^{\frac{3}{2}}}$, $\alpha=0$, $\beta=\tfrac92$ and $\gamma=\tfrac{11\pi^2}{36}$, and obtain
$$
U_3(\ell)
\sim
\tfrac{121\sqrt{33}\pi^3}{864}\,
\ell^{-3}
\exp\left(
\tfrac{\pi}{3}\sqrt{11\ell}
\right).
$$
Putting $\ell=3n$ leads us to
$$
U_3(3n)
\sim
\tfrac{121\sqrt{33}\pi^3}{23328}\,
n^{-3}
\exp\left(
\pi\sqrt{\tfrac{11n}{3}}
\right),
$$
as required.
\end{proof}
For a sufficiently large positive integer $\ell$, define
\begin{equation}\label{sellmiell}
s_\ell=\left\lfloor \ell-7\sqrt\ell\log\ell\right\rfloor
\quad\text{and}\quad
\mu_\ell=\left\lceil2\sqrt{s_\ell}\log s_\ell\right\rceil.
\end{equation}
Recall from Lemma~\ref{lm: restricted partitions} that
$\mathcal A(3,s_\ell)$ denotes the set of all partitions $\lambda$ such that
$|\lambda|\le s_\ell$ and every part of $\lambda$ belongs to
$\{3,4,\ldots,\mu_\ell\}$.

\begin{lm}\label{lm: restricted partitions p3}
We have
$
s_\ell\le \ell-3\mu_\ell
$
for all sufficiently large $\ell$, and
$$
\log|\mathcal A(3,s_\ell)|
=
\pi\sqrt{\tfrac23}\sqrt\ell+O(\log\ell).
$$
\end{lm}
\begin{proof}
Since $s_\ell\le \ell$, we have
$\mu_\ell\le 2\sqrt{\ell}\log \ell+1.$
Hence
$3\mu_\ell\le 6\sqrt{\ell}\log \ell+3
\le 7\sqrt{\ell}\log \ell$
for all sufficiently large $\ell$. Moreover,
$s_\ell\le \ell-7\sqrt{\ell}\log \ell$ by the definition of $s_\ell$.
Adding side by side, we obtain
$s_\ell+3\mu_\ell\le \ell.$ Therefore, $s_\ell\le  \ell-3\mu_\ell$ for all sufficiently large $\ell$.\\
Let $p_{\leq\mu}(s)$ denote the number of partitions of $s$ whose largest part is at most $\mu$. 
Since
$
\mu_\ell=\left\lceil2\sqrt{s_\ell}\log s_\ell\right\rceil=\mu(s_\ell),
$
Lemma~\ref{lm: restricted partitions}, applied with $t=s_\ell$, gives
$$
p_{\leq\mu_\ell}(s_\ell)\sim p(s_\ell).
$$
Moreover, applying the second part of Lemma~\ref{lm: restricted partitions} with $a=3$ and $t=s_\ell$, we obtain
$$
|\mathcal A(3,s_\ell)|
\geq
\frac{p(s_\ell)}{2(s_\ell+1)^2}
$$
for all sufficiently large $\ell$.\\
By the Hardy--Ramanujan formula,
$
\log p(s_\ell)
=
\pi\sqrt{\tfrac23}\sqrt{s_\ell}
+O(\log s_\ell).
$
Since
$$
|\mathcal A(3,s_\ell)|
\ge
\frac{p(s_\ell)}{2(s_\ell+1)^2},
$$
we obtain
$$
\log|\mathcal A(3,s_\ell)|
\ge
\log p(s_\ell)-2\log(s_\ell+1)-\log2
=
\pi\sqrt{\tfrac23}\sqrt{s_\ell}
+O(\log s_\ell).
$$
Since
$$
\sqrt{s_\ell}
=
\sqrt{\ell}
+O\!\left(\frac{\sqrt{\ell}\log\ell}{\sqrt{\ell}}\right)
=
\sqrt{\ell}
+O(\log\ell),
$$
it follows that
$$
\log|\mathcal A(3,s_\ell)|
\ge
\pi\sqrt{\tfrac23}\sqrt{\ell}
+O(\log\ell).
$$
Moreover,
$$
|\mathcal A(3,s_\ell)|
\le
\sum_{j=0}^{s_\ell}p(j)
\le
(s_\ell+1)p(s_\ell),
$$
and therefore
$$
\log|\mathcal A(3,s_\ell)|
\le
\log p(s_\ell)+\log(s_\ell+1)
=
\pi\sqrt{\tfrac23}\sqrt{s_\ell}
+O(\log s_\ell)
=
\pi\sqrt{\tfrac23}\sqrt{\ell}
+O(\log\ell).
$$
The lower and upper bounds complete the proof.
\end{proof}
The next lemma extends Lemma~\ref{lm: packing strips} to strips of different widths.
\begin{lm}\label{lm: unequal strips packing}
Let $c_1,c_2,\ldots,c_k$ and $\mu$ be positive integers such that
$$
\mu\le \min\{c_1,c_2,\ldots,c_k\}.
$$
Suppose we are given some integer-sided rectangles of height $1$, each of width at most $\mu$, and such that the sum of the widths of all rectangles is at most
$
\sum_{j=1}^{k}c_j-k\mu.
$
Then the~rectangles can be packed into horizontal strips of height $1$ and widths $c_1,c_2,\ldots,c_k$.
\end{lm}

\begin{proof}
Consider the rectangles in any fixed order. Place them from left to right in the first strip until the next rectangle does not fit, and then continue in the next strip. Suppose that a rectangle of width $b\le  \mu$ does not fit in a strip of width $c_j$ whose occupied width is $v$. Then $v+b>c_j$, and hence $v>c_j-b\geq c_j-\mu$. If more than $k$ strips were required, each of the first $k$ strips would have occupied width greater than $c_j-\mu$. The total width of all rectangles would then be greater than $\sum_{1\le  j\le  k}c_j-k\mu$, which is a contradiction.
\end{proof}

\begin{lm}\label{lm: lower construction p3}
For sufficiently large $n$, put
$$
i_n=\left\lfloor\tfrac{2n}{11}\right\rfloor,
\quad
r_n=n-i_n,
\quad
w_n=\left\lfloor\tfrac{r_n}{2}\right\rfloor
\quad\text{and}\quad
\ell_n=3r_n-2w_n.
$$
Then
$$
p(3,n)
\ge a_3(i_n)a_3(w_n)|\mathcal{A}(3,s_{\ell_n})|.
$$
\end{lm}

\begin{proof}
Choose a partition $\lambda_3$ of $i_n$ into parts at least $3$. Replace every part $b$ of $\lambda_3$ by a~rectangle $3\times b$. Since the parts have sum $i_n$, these rectangles tile a strip of height $3$ and width $i_n$. The remaining area is a rectangle of height $3$ and width $r_n$.\\
Further, choose independently a partition $\lambda_2$ of $w_n$ into parts at least $3$. Replace every part $b$ of $\lambda_2$ by a rectangle $2\times b$. These rectangles tile a strip of height $2$ and width $w_n$ inside the remaining rectangle. 
The remaining area consists of three horizontal strips of dimensions
$$
1\times r_n,\quad
1\times(r_n-w_n),
\quad\text{and}\quad
1\times(r_n-w_n).
$$
For an illustration of the idea, we refer the reader to Figure~\ref{fig: strips p3}.
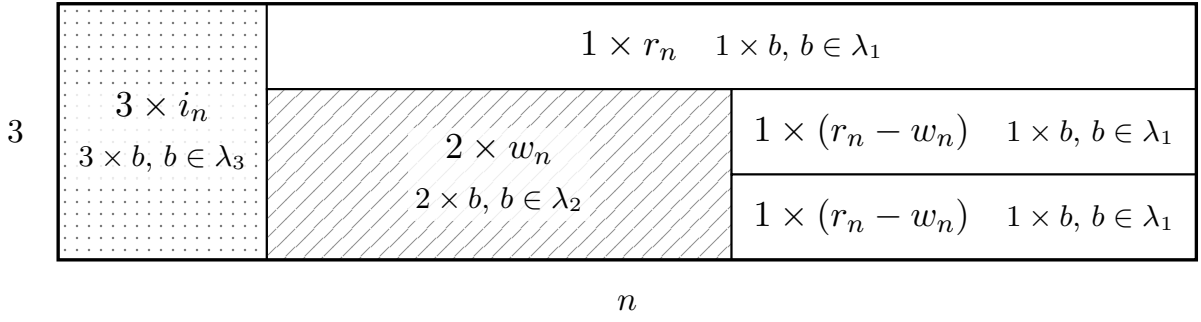
\begin{figure}[ht]
\centering
\resizebox{\textwidth}{!}{%
\begin{tikzpicture}[x=1cm,y=1cm,
  zone/.style={draw,line width=.6pt,align=center,font=\small},
  dotted/.style={pattern={Dots[distance=4.5pt,radius=.45pt]},pattern color=black!55},
  hatched/.style={pattern={Lines[distance=5pt,angle=45,line width=.35pt]},pattern color=black!45}]
\def\xa{0}   \def\xb{2.2} \def\xc{7.1} \def\xd{12}
\def\ya{0}   \def\yb{0.9} \def\yc{1.8} \def\yd{2.7}
\filldraw[zone,dotted] (\xa,\ya) rectangle (\xb,\yd);
\node[align=center,font=\small,fill=white,fill opacity=.85,text opacity=1,inner sep=2pt]
  at ({(\xa+\xb)/2},{(\ya+\yd)/2}) {$3\times i_n$\\[2pt]\scriptsize $3\times b$,\ $b\in\lambda_3$};
\filldraw[zone,fill=white] (\xb,\yc) rectangle (\xd,\yd);
\node[font=\small] at ({(\xb+\xd)/2},{(\yc+\yd)/2})
  {$1\times r_n$\quad\scriptsize $1\times b$,\ $b\in\lambda_1$};
\filldraw[zone,hatched] (\xb,\ya) rectangle (\xc,\yc);
\node[align=center,font=\small,fill=white,fill opacity=.85,text opacity=1,inner sep=2pt]
  at ({(\xb+\xc)/2},{(\ya+\yc)/2}) {$2\times w_n$\\[2pt]\scriptsize $2\times b$,\ $b\in\lambda_2$};
\filldraw[zone,fill=white] (\xc,\yb) rectangle (\xd,\yc);
\node[font=\small] at ({(\xc+\xd)/2},{(\yb+\yc)/2}) {$1\times(r_n-w_n)$\quad\scriptsize $1\times b$,\ $b\in\lambda_1$};
\filldraw[zone,fill=white] (\xc,\ya) rectangle (\xd,\yb);
\node[font=\small] at ({(\xc+\xd)/2},{(\ya+\yb)/2}) {$1\times(r_n-w_n)$\quad\scriptsize $1\times b$,\ $b\in\lambda_1$};
\draw[line width=1pt] (\xa,\ya) rectangle (\xd,\yd);
\node[font=\footnotesize] at (-0.45,{(\ya+\yd)/2}) {$3$};
\node[font=\footnotesize] at ({(\xa+\xd)/2},-0.45) {$n$};
\end{tikzpicture}%
}
\caption{Decomposition of the rectangle $3\times n$ used in the proof of
Lemma~\ref{lm: lower construction p3}. The dotted block of height $3$ is tiled by
the rectangles coming from $\lambda_3$, the hatched block of height $2$ by those
coming from $\lambda_2$, and the three white strips of height $1$ receive the
rectangles coming from $\lambda_1$ together with the unit squares.}
\label{fig: strips p3}
\end{figure}
If these strips are placed side by side, they form a strip of length
$$
r_n+(r_n-w_n)+(r_n-w_n)=3r_n-2w_n=\ell_n.
$$
Choose a partition $\lambda_1\in\mathcal A(3,s_{\ell_n})$ and replace every part $b$ of $\lambda_1$ by a rectangle $1\times b$. By the definition of $\mathcal A(3,s_{\ell_n})$, the sum of the parts of $\lambda_1$ is at most $s_{\ell_n}$, so the corresponding rectangles have total width at most $s_{\ell_n}$. Moreover, every rectangle has width at most $\mu_{\ell_n}$.\\
By Lemma~\ref{lm: restricted partitions p3},
$$
|\lambda_1|\le s_{\ell_n} \le \ell_n-3\mu_{\ell_n}.
$$
Since
$
r_n-w_n=r_n-\left\lfloor\frac{r_n}{2}\right\rfloor
\ge \frac{r_n}{2}=\tfrac{1}{2}(n-i_n) \ge \frac{9n}{22}
$ 
and
$
\mu_{\ell_n}=O(\sqrt n\log n),
$ 
we have\\
$
\mu_{\ell_n}\le r_n-w_n
$
for all sufficiently large $n$. Thus, every rectangle has width not exceeding the width of any of the three strips. Therefore, Lemma~\ref{lm: unequal strips packing} shows that the corresponding rectangles can be packed into the three strips.\\
Finally, tile every remaining unit cell by a square $1\times1$. The resulting multiset of rectangles forms a~tiling of the rectangle $3\times n$.\\
The construction is injective. The added rectangles are squares $1\times1$, while every selected rectangle has its larger side at least $3$.\\
After removing all unit squares, every remaining rectangle has smaller side equal to $1$, $2$, or $3$. Hence, its smaller side uniquely determines whether the rectangle was obtained from $\lambda_1$, $\lambda_2$, or $\lambda_3$.\\
The length of its larger side determines the corresponding part of $\lambda_1$, $\lambda_2$, or $\lambda_3$. Therefore, the three selected partitions can be recovered from the resulting multiset. Since the choices are independent, the number of constructed rectangular partitions is at least $a_3(i_n)a_3(w_n)|\mathcal A(3,s_{\ell_n})|$.
\end{proof}
\begin{re}
A similar construction can be carried out for the rectangle $4\times n$. One first uses rectangles of height $4$, then rectangles of height $3$, then rectangles of height $2$, and finally rectangles of height $1$ in the remaining horizontal strips. The corresponding asymptotic calculations then give
$$
\log p(4,n)
\ge
\pi\sqrt{\tfrac{8}{3}\,H_4\,n}+o(\sqrt n)
=
\tfrac{5\pi\sqrt2}{3}\sqrt n+o(\sqrt n),
$$
which agrees with the conjectured asymptotic formula~\eqref{conjH} for $m=4$.\\
For $5\times n$, however, the same staircase construction no longer gives the conjectured exponential term. The optimal widths of the levels $5$, $4$, $3$, and $2$ do not fit side by side inside a rectangle of width $n$. In particular, it means that for $5\times n$ a different arrangement of the strips than in the cases $3\times n$ and $4\times n$ seems to be needed.
\end{re}
\begin{proof}[Proof of Theorem \ref{thm: asymptotic p(3,n)}]
Lemma~\ref{lm: unrestricted p3} guarantees that
$$
\log p(3,n)
\le 
\pi\sqrt{\tfrac{11n}{3}}
-3\log n
+O(1).
$$
We now prove the corresponding lower bound. We use the parameters from Lemma~\ref{lm: lower construction p3}. Theorem~\ref{thma3} shows that
$
\log a_3(k)=\pi\sqrt{\tfrac23}\sqrt k+O(\log k)
$
as $k\to\infty$.\\ Lemma~\ref{lm: restricted partitions p3} gives
$
\log|\mathcal A(3,s_{\ell_n})|=\pi\sqrt{\tfrac23}\sqrt{\ell_n}+O(\log n).
$\\
Therefore, by Lemma~\ref{lm: lower construction p3},
\begin{align*}
\log p(3,n)\ge\pi\sqrt{\tfrac23}
\left(\sqrt{i_n}+\sqrt{w_n}+\sqrt{\ell_n}\right)
+O(\log i_n)+O(\log w_n)+O(\log\ell_n).
\end{align*}
Since
$
\log i_n=\log n+O(1),$
 $\log w_n=\log n+O(1),$
 $\log\ell_n=\log n+O(1),$
we obtain
$$
\log p(3,n)
\ge
\pi\sqrt{\tfrac23}
\left(
\sqrt{i_n}
+
\sqrt{w_n}
+
\sqrt{\ell_n}
\right)
+
O(\log n).
$$
We have
$$
i_n=\tfrac{2n}{11}+O(1),\;
r_n=\tfrac{9n}{11}+O(1),\;
w_n=\tfrac{r_n}{2}+O(1)=\tfrac{9n}{22}+O(1),\;
\ell_n=2r_n+O(1)=\tfrac{18n}{11}+O(1).
$$
Consequently,
$$
\sqrt{i_n}
=
\sqrt{\tfrac{2}{11}}\sqrt n+O(n^{-\frac{1}{2}}),\;
\sqrt{w_n}
=
\sqrt{\tfrac{9}{22}}\sqrt n+O(n^{-\frac{1}{2}}),\;
\sqrt{\ell_n}
=
\sqrt{\tfrac{18}{11}}\sqrt n+O(n^{-\frac{1}{2}}).
$$
Hence,
$$
\sqrt{i_n}+\sqrt{w_n}+\sqrt{\ell_n}
=
\left(
\sqrt{\tfrac{2}{11}}
+
\sqrt{\tfrac{9}{22}}
+
\sqrt{\tfrac{18}{11}}
\right)\sqrt n
+
O(n^{-\frac{1}{2}})
=
\sqrt{\tfrac{11}{2}}\sqrt n
+
O(n^{-\frac{1}{2}}).
$$
Therefore
$$
\log p(3,n)\ge\pi\sqrt{\tfrac{11n}{3}}+O(\log n).
$$
Combining the lower and upper estimates gives
$$
\log p(3,n)=\pi\sqrt{\tfrac{11n}{3}}+O(\log n).
$$
Equivalently,
$p(3,n)=\exp\left(\pi\sqrt{\tfrac{11n}{3}}+O(\log n)\right),
$ as required.
\end{proof}

\appendix

\section{A guide to the partition functions used in the paper}
For the convenience of the reader, we collect the main functions used in the paper, together with their combinatorial interpretations and initial values.

The number $D(\ell)$ counts multisets of ordered rectangle types $(a,b)$ with total area $\ell$. The types $(a,b)$ and $(b,a)$ are different when $a\neq b$. Equivalently, $D(\ell)$ counts colored partitions of $\ell$ in which a part of size $s$ is available in $d(s)$ colors indexed by the ordered pairs $(a,b)$ satisfying $ab=s$. We have $\sum_{\ell\geq0}D(\ell)x^\ell=\prod_{s\geq1}(1-x^s)^{-d(s)}$ and
$
(D(0),D(1),\ldots,D(6))=(1,1,3,5,11,17,34).
$

The number $R(\ell)$ counts multisets of rectangle types with total area $\ell$, where rotations are identified. We have $\sum_{\ell\geq0}R(\ell)x^\ell=\prod_{s\geq1}(1-x^s)^{-\rho(s)}$ and
$
(R(0),R(1),\ldots,R(6))=(1,1,2,3,6,8,14).
$

The number $P_v(\ell)$ counts the multisets considered in $R(\ell)$ with the additional restriction $1\le  a\le  b\le  v$. We have $\sum_{\ell\geq0}P_v(\ell)x^\ell=\prod_{1\le  a\le  b\le  v}(1-x^{ab})^{-1}$. For example,
$
(P_2(0),\ldots,P_2(6))=(1,1,2,2,4,4,6)
$
and
$
(P_3(0),\ldots,P_3(6))=(1,1,2,3,5,6,10).
$

The number $p(m,n)$ counts multisets of rectangle types that can be arranged to cover the rectangle $m\times n$. Rotations are identified, and different arrangements of the same multiset are not counted separately. We exhibit the first few values of the diagonal case $p(n,n)$: 
$
(p(1,1),p(2,2),\ldots,p(7,7))
=
(1,4,21,192,2035,27407,399618).
$

The number $Q(\ell)$ counts partitions of $\ell$ into perfect squares, or equivalently multisets of squares with total area $\ell$. We have that $\sum_{\ell\geq0}Q(\ell)x^\ell=\prod_{j\geq1}(1-x^{j^2})^{-1}$ and
$
(Q(0),Q(1),\ldots,Q(6))=(1,1,1,1,2,2,2).
$

At the end, we point out that if $h\le  w$, then
$
p(h,w)\le  P_w(hw)\le  R(hw)\le  D(hw)
$
and
$
Q(\ell)\le  R(\ell).
$

\section*{Acknowledgments}
The authors would like to thank Jakub Byszewski for suggesting the direction pursued in this paper. K.G. is grateful to Cyril Banderier, Jehanne Dousse, Michael Wallner and Nian Hong Zhou for inspiring discussions on possible new directions in the theory of rectangle partitions during the 10th International Conference on Lattice Path Combinatorics and Applications, held in Vienna in 2026. K.G. was supported by the National Science Center grant no.~2024/53/N/ST1/01538.

M.Z. would like to thank Daria Kuliś, MD, for performing a successful operation.



\end{document}